\documentclass[11pt]{article}
\usepackage{etex}           %
\usepackage{amsmath}        
\usepackage{amssymb}        
\usepackage{amsthm}         
\usepackage{mathtools}      
\usepackage{enumitem} 

    \usepackage{tikz}
    \usetikzlibrary{decorations.pathreplacing} 
    \usetikzlibrary{math} 
    \usepackage{colortbl} 
        \usepackage{float}
        \usepackage{subcaption}
        \counterwithin{figure}{section}
        \counterwithin{table}{section}
        \usepackage{booktabs}
        \usepackage[pdfborder={0 0 0}]{hyperref} 
        \usepackage{cleveref}
        \usepackage{autonum}

    \allowdisplaybreaks

    \usepackage{parskip}
    \makeatletter
    \renewenvironment{abstract}{%
      \if@twocolumn
        \section*{\abstractname}%
      \else
        \small
        \begin{center}%
          {\bfseries \abstractname\vspace{-.5em}\vspace{\z@}}%
        \end{center}%
        \quotation
        \noindent\ignorespaces
      \fi
    }{%
      \if@twocolumn\else\endquotation\fi
    }
    \makeatother
    \usepackage{geometry}[2010/09/12]
        \newgeometry{
            top=1in,
            left=1in,
            bottom=0.7in,
            right=1in,
            includefoot   
    }

    \newcommand{\lrcornerqed}[0]{}
    \renewcommand{\lrcornerqed}[0]{\hfill $\lrcorner$} 
    \newcommand{\lrcornerqedequation}[0]{
        \\[-3.4em]
        \phantom{.}
        \lrcornerqed
        \\[-0.3em]
        }   
    \usepackage{bbm}
            
    \newtheorem{thm}{Theorem}[section] 
    
    \newtheorem{cor}[thm]{Corollary}
    \newtheorem{prop}[thm]{Proposition}

    \theoremstyle{definition} 
    \newtheorem{definition}[thm]{Definition}
    \newtheorem{example}[thm]{Example}
    \numberwithin{equation}{section}
    \newtheorem{remark}[thm]{Remark}
    \newtheorem{obs}[thm]{Observation}

    \crefname{lemma}{Lemma}{Lemmas}
    \crefname{thm}{Theorem}{Theorems}
    \crefname{prop}{Proposition}{Propositions}
    \crefname{definition}{Definition}{Definitions}
    \crefname{example}{Example}{Examples}
    \crefname{claim}{Claim}{Claims}
    \crefname{conj}{Conjecture}{Conjectures}
    \crefname{cor}{Corollary}{Corollaries}
    \crefname{figure}{Figure}{Figures}
    \crefname{remark}{Remark}{Remarks}
    \crefname{table}{Table}{Tables}
    \crefname{section}{Section}{Sections}
    \crefname{chapter}{Chapter}{Chapters}
    \crefname{appendix}{Appendix}{Appendices}
    \crefname{obs}{Observation}{Observations}

\usepackage{xcolor}

\newcommand{\myatop}[2]{\genfrac{}{}{0pt}{}{#1}{#2}}

\title{Combinatorial enumeration of lattice paths by flaws with respect to a linear boundary of rational slope}
\author{Federico Firoozi, Jonathan Jedwab, Amarpreet Rattan \\
Department of Mathematics, Simon Fraser University, Burnaby BC, Canada}

\begin{document}
\date{29 May 2024 (revised 15 July 2025)}

\maketitle


\begin{abstract}
    Let $a,b$ be fixed positive coprime integers. For a positive integer~$g$,
    write $W_k(g)$ for the set of lattice paths from the startpoint $(0,0)$ to the endpoint $(ga,gb)$ with steps restricted to $\{(1,0), (0,1)\}$, having exactly $k$ flaws (lattice points lying above the linear boundary connecting the startpoint to the endpoint). 
    We determine~$|W_k(g)|$ for all $k$ and~$g$.
    The enumeration of lattice paths with respect to a linear boundary while accounting for flaws has a long and rich history, dating back at least to the 1949 results of Chung and Feller. 
    The only previously known values of $|W_k(g)|$ are the extremal cases $k = 0$ and $k = g(a+b)-1$, determined by Bizley in 1954. 
    Our main combinatorial result is that a certain subset of $W_k(g)$ is in bijection with~$W_{k+1}(g)$. One consequence is that the value $|W_k(g)|$ is constant over each successive set of $a+b$ values of~$k$. This in turn allows us to derive a recursion for $|W_k(g)|$ whose base case is given by Bizley's result for $k=0$. We solve this recursion to obtain a closed form expression for~$|W_k(g)|$ for all $k$ and~$g$. Our methods are purely combinatorial.
\end{abstract}

\section{Introduction}
\label{sec:intro}
The lattice path shown in \cref{fig:sample} contains exactly five lattice points that lie above the linear boundary joining the startpoint $(0,0)$ to the endpoint $(8,6)$.

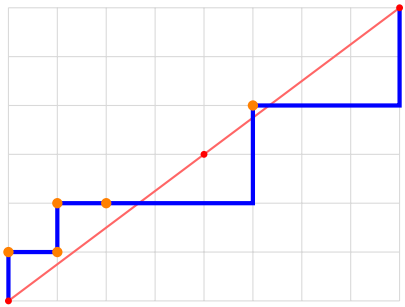
\begin{figure}[H]
    \begin{center}
    \begin{tikzpicture}[scale = 0.65]
    \draw[gray!30] (0,0) grid (8,6);
    \draw[thick,red!60] (0,0) -- (8,6);
    \draw[ultra thick, blue] (0,0)  -- (0,1) -- (1,1) -- (1,2) -- (2,2) -- (3,2) -- (4,2) -- (5,2) -- (5,3) -- (5,4) -- (6,4) -- (7,4) -- (8,4) -- (8,5) -- (8,6);
    \fill[red]	    (0,0) circle (2pt)
    			(4,3) circle (2pt)
    			(8,6) circle (2pt);
    \fill[orange]	    (0,1) circle (3pt)
    			(1,1) circle (3pt)
    			(1,2) circle (3pt)
                    (2,2) circle (3pt)
    			(5,4) circle (3pt);
    \end{tikzpicture}
    \caption{Lattice path from $(0,0)$ to $(8,6)$ with five lattice points above the line connecting $(0,0)$ to $(8,6)$.}
    \label{fig:sample}
    \end{center}
\end{figure}

Throughout, $a,b$ are fixed positive coprime integers and $g$ is a positive integer.
Our objective is to count the number of lattice paths from the startpoint $(0,0)$ to the endpoint $(ga,gb)$ with steps  restricted to $\{(1,0), (0,1)\}$, having exactly $k$ lattice points lying above the linear boundary joining the startpoint to the endpoint. 

Let $p$ be a path.
The \textit{boundary} of $p$ is the line joining its startpoint to its endpoint. 
The path $p$ \textit{contains} the lattice point $(x+i,y+j)$ (equivalently,
$(x+i,y+j)$ is \textit{a point of}~$p$)
if $p$ starts at~$(x,y)$, and the first $i+j\ge 0$ steps of $p$ consist of $i$ of the $(1,0)$ steps and $j$  of the $(0,1)$ steps (in any order). 
We consider the points of $p$ to be ordered according to increasing values of~$i+j$.
A point of $p$ is a \textit{flaw} if it lies strictly above the boundary of~$p$.
For example, the path in \cref{fig:sample} has the five flaws $(0,1),(1,1),(1,2),(2,2),(5,4)$ denoted in orange.

\begin{definition}[Sets $W(g)$ and $W_k(g)$]\label{def:path_sets}
    Let $W(g)$ be the set of all paths from $(0,0)$ to $(ga,gb)$, and let $W_k(g)$ be the subset of such paths having exactly $k$ flaws. \lrcornerqed
\end{definition}

The possible values for the number $k$ of flaws of a path are those satisfying $0\le k<g(a+b)$. 
Straightforward counting shows that $|W(g)| = \binom{ga+gb}{ga}$.
The central objective of this paper is to find an explicit formula for~$|W_k(g)|$ for all $g,k$ satisfying $0\le k<g(a+b)$.  The extremal values $|W_0(g)|$ and $|W_{g(a+b)-1}(g)|$ were found by Bizley~\cite{biz} in 1954 (see \cref{thm:actual_bizley} below). 
Until now, the value of $|W_k(g)|$ was unknown for all other~$k$.

When $g=1$, the values $|W_k(g)|$ and $|W_{g(a+b)-1-k}(g)|$ are equal: in this case, a path cannot contain lattice points on the boundary other than its startpoint and endpoint, so rotation of the path through $180^\circ$ bijectively maps the set $W_k(g)$ to the set $W_{g(a+b)-1-k}(g)$. 
However, in the case $g>1$, a path can contain such lattice points and, 
because points on the boundary are not counted as flaws, rotation of the path through $180^\circ$ does not map the set $W_k(g)$ to the set $W_{g(a+b)-1-k}(g)$. In fact, the values $|W_k(g)|$ and $|W_{g(a+b)-1-k}(g)|$ are not equal in general.

\cref{table:computer_enumeration} displays the numerical value of $|W_k(4)|$ for $(a,b)=(3,2)$, obtained by computer evaluation.
We note two apparent properties suggested by these values:

\begin{table}[ht]
\definecolor{row1}{rgb}{0.95,0.95,0.95}
\definecolor{row2}{rgb}{0.85,0.85,0.85}
\centering
\begin{tabular}[b]{|c|c|c|}
\hline
\rowcolor{red!60} $k$ & $|W_k(4)|$ & $|W_k(4)| - |W_{k+1}(4)|$ \\
\specialrule{1pt}{0pt}{0pt}
\rowcolor{row1} 0 & 7229 & 0 \\
\hline
\rowcolor{row1} 1 & 7229 & 0 \\
\hline
\rowcolor{row1} 2 & 7229 & 0 \\
\hline
\rowcolor{row1} 3 & 7229 & 0 \\
\hline
\rowcolor{row1} 4 & 7229 & 754 \\
\hline
\rowcolor{row2} 5 & 6475 & 0 \\
\hline
\rowcolor{row2} 6 & 6475 & 0 \\
\hline
\rowcolor{row2} 7 & 6475 & 0 \\
\hline
\rowcolor{row2} 8 & 6475 & 0 \\
\hline
\rowcolor{row2} 9 & 6475 & 437 \\
\hline
\rowcolor{row1} 10 & 6038 & 0 \\
\hline
\rowcolor{row1} 11 & 6038 & 0 \\
\hline
\rowcolor{row1} 12 & 6038 & 0 \\
\hline
\rowcolor{row1} 13 & 6038 & 0 \\
\hline
\rowcolor{row1} 14 & 6038 & 586 \\
\hline
\rowcolor{row2} 15 & 5452 & 0 \\
\hline
\rowcolor{row2} 16 & 5452 & 0 \\
\hline
\rowcolor{row2} 17 & 5452 & 0 \\
\hline
\rowcolor{row2} 18 & 5452 & 0 \\
\hline
\rowcolor{row2} 19 & 5452 &\multicolumn{1}{c}{ \cellcolor{white}}\\
\cline{1-2}
\end{tabular}
\caption{Computer evaluation of $|W_k(4)|$ for $(a,b)= (3,2)$.}
\label{table:computer_enumeration}
\end{table}

\begin{enumerate}[label=P\arabic*]
    \item \label{property:const_blocks}
    (\textbf{Constant on blocks}). The value $|W_k(g)|$ is constant on each of $g$ distinct ``blocks'' of $a+b$ consecutive values of~$k$.

    \item \label{property:strict_decrease}
    (\textbf{Strictly decreasing}). The value $|W_k(g)|$ is strictly decreasing between successive blocks.
\end{enumerate}
We shall show that properties \ref{property:const_blocks} and \ref{property:strict_decrease} both hold for all values of~$g,a,b$.

\cref{table:computer_enumeration}, in addition to displaying the value of $|W_k(4)|$ for $(a,b)=(3,2)$, also displays the value of the difference $|W_k(4)|-|W_{k+1}(4)|$. 
These differences suggest a strategy for achieving our central objective: identify a subset $S_k(g)$ of $W_k(g)$ having cardinality $|W_k(g)|-|W_{k+1}(g)|$, and show that the sets $W_k(g) \setminus S_k(g)$ and $W_{k+1}(g)$ are in bijection. We achieve this in our main combinatorial result (\cref{thm:set_sizes}).
Properties~\ref{property:const_blocks} and \ref{property:strict_decrease} 
follow as consequences of this result.

We introduce some additional vocabulary before defining the subset~$S_k(g)$.

\begin{definition}[Path concatenation]\label{def:concatenation}
    Let $p_1$ and $p_2$ be paths having  arbitrary startpoints. The \textit{path concatenation} $p_1p_2$ is the path that starts at the startpoint of $p_1$, takes all the (ordered) steps of~$p_1$, and then takes all the (ordered) steps of~$p_2$. \lrcornerqed
\end{definition}

\begin{definition}[Boundary points]\label{def:BP}
    The \textit{boundary points} of a path $p \in W(g)$ are the points of $p$ that lie on its boundary. Boundary points of $p$ other than the startpoint $(0,0)$ and endpoint $(ga,gb)$ are \textit{interior boundary points} of~$p$.   The startpoint and interior boundary points are the \textit{non-terminal} boundary points. \lrcornerqed
\end{definition}
    
The lattice points lying on the boundary joining $(0,0)$ to $(ga,gb)$ are the $g+1$ points of the form $(ja,jb)$ for $0 \le j \le g$ (see \cref{fig:boundary_points}).
The number of interior boundary points of a path $p \in W(g)$ therefore lies 
in $\{0,1,\dots,g-1\}$. 

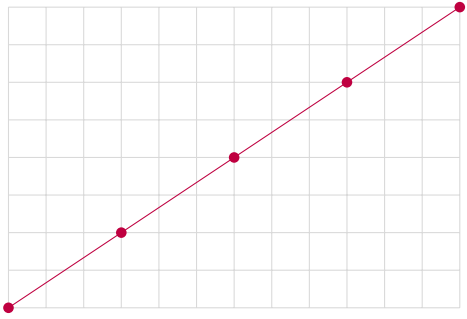
\begin{figure}
    \centering
    \begin{tikzpicture}[scale=0.5]
        \begin{scope}[shift = {(-2,-12)}]
            \tikzmath{
            int \a, \b, \g;
            \a = 3;
            \b = 2;
            \g = 4;
            } 
            \draw[gray!30] 
                (0,0) grid (\g*\a, \g*\b);
            \draw[purple] 
                (0,0) -- (\g*\a, \g*\b); 
            \foreach \j in {0,...,\g}{
                \fill[purple]
                    (\j*\a, \j*\b) circle (4pt);
            }
        \end{scope}
    \end{tikzpicture}
    \caption{Let $(a,b) = (3,2)$. The boundary of a path whose endpoint is $(ga,gb) = (4 \cdot 3, 4\cdot 2)$ contains $g+1 = 5$ lattice points \textcolor{purple}{(red vertices)}.}
    \label{fig:boundary_points}
\end{figure}

Recall that the number $k$ of flaws of a path in $W(g)$ satisfies $0\le k<g(a+b)$. A path in $W(g)$ containing $g(a+b)-1$ flaws has \textit{max flaws}.  Equivalently, the set of paths with max flaws is $W_{g(a+b)-1}(g)$.

\begin{definition}[Subset $S_k(g)$]
\phantomsection\label{def:S} For $0 \leq k < g(a+b)-1$,  let $S_k(g)$ be the 
set of \textit{min-max paths}, namely the subset of $W_k(g)$ containing all 
paths of the form $p_1p_2$, where
    $p_1 \in W_0(g-j)$ and $p_2 \in W_k(j)$  for some $j$ satisfying $0<j<g$, and $p_2$ has max flaws.
    We write $S(g) \coloneqq  \bigcup_{k} S_k(g)$. 
    \lrcornerqed
\end{definition}

See \cref{fig:s_path_examples} for two example min-max paths in~$S(4)$.
A min-max path is, for some~$j$, the concatenation of a path $p_1$ from $(0,0)$ to $\big((g-j)a,(g-j)b\big)$ having no flaws with a path $p_2$ from 
$(0,0)$ to $(ja,jb)$ having max flaws. 
The condition that $p_2$ has max flaws implies that $S_k(g)$ is empty unless $k=j(a+b)-1$ for some $j$ satisfying $0 < j < g$ (and in particular $S_0(g)$ is empty).
So we have
\begin{equation}
    S_k(g) = \varnothing \quad \mbox{for $k \not\equiv -1\pmod{a+b}$}, \label{eqn:S1}
\end{equation}
and, for each $j$ satisfying $0 < j < g$, 
\begin{equation}
    S_{j(a+b)-1}(g) = \Big\{p_1 p_2 : p_1 \in W_0(g-j) \mbox{ and } p_2 \in W_{j(a+b)-1}(j)\Big\} \label{eqn:S2}.
\end{equation}
Note that the path $p_2$ in \cref{def:S} does not contain an interior boundary point (because it has max flaws), but the path $p_1$ might (see \cref{fig:s_path_examples}).

\begin{figure}[H]
\definecolor{BurntOrange}{RGB}{235, 119, 52}
\centering
\begin{subfigure}{.4\textwidth}
  \centering
    \begin{tikzpicture}[scale = 0.5]
        \draw[black,line width = 0.8pt, dotted]
            (6-3*2/13, 4+3*3/2*2/13) -- (6+3*2/13, 4-3*3/2*2/13);
        \draw[gray!30] (0,0) grid (12,8);
        \draw[thick,red!60] (0,0) -- (12,8);
        \draw[ultra thick, violet] (0,0)  -- (1,0) -- (2,0) -- (3,0) -- (4,0) -- (4,1) -- (4,2) -- (5,2) -- (5,3) -- (6,3) -- (6,4);
        \draw[ultra thick, BurntOrange] (6,4) -- (6,5) -- (7,5) -- (7,6) -- (7,7) -- (8,7) -- (9,7) -- (9,8) -- (10,8) -- (11,8) -- (12,8);
        \draw[shift = {(4.5,1.5)}, color = violet] node[] {$p_1$};
        \draw[shift = {(7.5,7.5)}, color = BurntOrange] node[] {$p_2$};
        \fill[red]	(0,0) circle (3pt)
        			(3,2) circle (3pt)
        			(6,4) circle (3pt)
        			(9,6) circle (3pt)
        			(12,8) circle (3pt);
    \end{tikzpicture}
  \caption{A path $p_1p_2$ in $S_9(4)$, where ${p_1 \in W_0(2)}$ and ${p_2 \in W_9(2)}$.}
  \label{fig:s_path_1}
\end{subfigure}%
\qquad
\begin{subfigure}{.4\textwidth}
  \centering
    \begin{tikzpicture}[scale = 0.5]
        \draw[black,line width = 0.8pt, dotted]
            (9-3*2/13, 6+3*3/2*2/13) -- (9+3*2/13, 6-3*3/2*2/13);
        \draw[gray!30] (0,0) grid (12,8);
        \draw[thick,red!60] (0,0) -- (12,8);
        \draw[ultra thick, violet] (0,0)  -- (1,0) -- (2,0) -- (2,1) -- (3,1) -- (3,2) -- (4,2) -- (5,2) -- (6,2) -- (6,3) -- (7,3) -- (8,3) -- (8,4) -- (8,5) -- (9,5) -- (9,6);
        \draw[ultra thick, BurntOrange]  (9,6) -- (9,7) -- (10,7) -- (10,8) -- (11,8) -- (12,8);
        \draw[shift = {(5.5,1.5)}, color = violet] node[] {$p_1$};
        \draw[shift = {(9.5,7.5)}, color = BurntOrange] node[] {$p_2$};
        \fill[red]	(0,0) circle (3pt)
        			(3,2) circle (3pt)
        			(6,4) circle (3pt)
        			(9,6) circle (3pt)
        			(12,8) circle (3pt);
    \end{tikzpicture}
  \caption{A path $p_1p_2$ in $S_4(4)$, where  \mbox{$p_1 \in W_0(3)$} and ${p_2 \in W_4(1)}$.}
  \label{fig:s_path_2}
\end{subfigure}
\caption{Two example min-max paths in $S(4)$ for $(a,b)= (3,2)$.}
\label{fig:s_path_examples}
\end{figure}
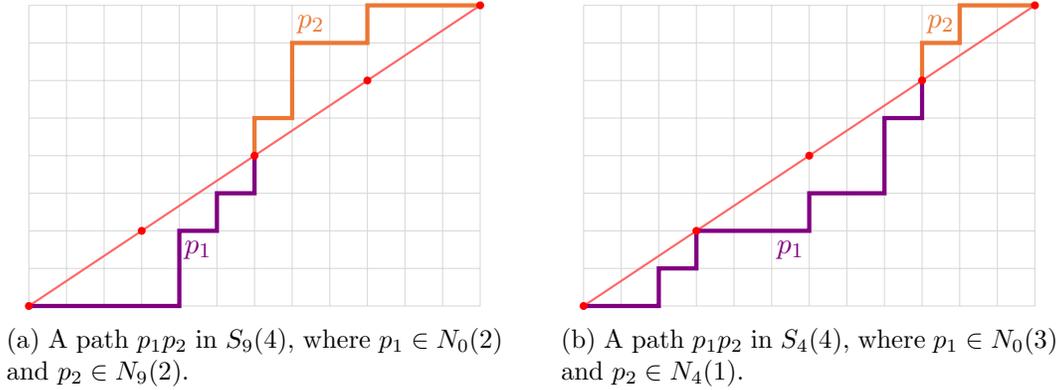

\subsection{Main combinatorial result and consequences}

\begin{thm}[Main combinatorial result]
    \label{thm:set_sizes}
    Let $g,k$ satisfy $0\le k <g(a+b)-1$. Then
    $$
        |W_k(g) \setminus S_k(g)| = |W_{k+1}(g)|.
    $$
    \\[-3.4em] 
    \phantom{.}
\end{thm}
We shall prove \cref{thm:set_sizes} in \cref{sec:bijection} combinatorially. 
Define  
    \begin{align}
        \label{def:mu_j}
        \mu_j (g) \coloneqq |W_{j(a+b)}(g)|
        \quad \text{for each $j$ satisfying $0\le j<g$}.
   \end{align} 
The following result is a first consequence of \cref{thm:set_sizes} and establishes property~\ref{property:const_blocks}.

\begin{cor}[Constant on blocks]
\phantomsection
\label{cor:blocks}
    Let $g$ be a positive integer.  Then
    \[
    |W_k(g)| = \mu_j(g) \ \ \text{for all $j,k$ satisfying $0\le j<g$ and $j(a+b)\le k < (j+1)(a+b)$}.
    \]
    \\[-3.4em] 
    \phantom{.}
\end{cor}

\begin{proof}
The result follows directly from \cref{thm:set_sizes} and~\eqref{eqn:S1}.
\end{proof}
We now observe two further consequences of \cref{thm:set_sizes}. 

\begin{cor}[Recurrence relation]
\label{cor:rec}
    Let $g$ be a positive integer. Then
    \begin{equation}
    \phantomsection
        \label{eq:rec}
        \mu_{j-1}(g)-\mu_0(g-j) \mu_{j-1}(j) =\mu_{j}(g) \ \ \text{for each $j$ satisfying $0<j< g$}.
    \end{equation}
    \\[-3.4em] 
    \phantom{.}
\end{cor}

\begin{proof}
Let $j$ satisfy $0<j< g$ and let $k=j(a+b)-1$. 
Since $S_k(g)$ is a subset of $W_k(g)$, we have by \cref{thm:set_sizes} that 
\begin{align}
\label{eq:cor:rec}
    |W_k(g)| - |S_k(g)| = |W_{k+1}(g)|.
\end{align}

We know from \cref{cor:blocks} that $|W_k(g)| = \mu_{j-1}(g)$ and $|W_{k+1}(g)| = \mu_j(g)$, and from \eqref{eqn:S2} and \cref{cor:blocks} that
$|S_k(g)| = |S_{j(a+b)-1}(g)| = \mu_0 (g-j) \mu_{j-1}(j)$. Substitute these values into \eqref{eq:cor:rec} to obtain the result.
\end{proof}
The next corollary establishes property \ref{property:strict_decrease}.

\begin{cor}[Strictly decreasing]\label{cor:decreasing}
        Let $g$ be a positive integer. Then $\mu_0(g) > \mu_{1}(g) > \cdots > \mu_{g-1}(g).$
\end{cor}
\begin{proof}
This follows from \cref{cor:rec}, noting that 
$\mu_0(g-j)\mu_{j-1}(j)>0$
for $0 < j < g$ by the definition of $\mu_j(g)$ given in~\eqref{def:mu_j}.
\end{proof}

\subsection{The value of \texorpdfstring{$\mu_j(g)$}{mu\_j(g)} and the main enumerative result}\label{subsec:BG:the_value_of_mu_j}

Recall that our central objective is to find an explicit formula for~$|W_k(g)|$ for all $g,k$ satisfying $0 \le k < g(a+b)$.  By \cref{cor:blocks}, it is sufficient to determine the values~$\mu_j(g)$.

The recurrence relation of \cref{cor:rec} for $\mu_j(g)$ has a unique solution for each $j,g$ satisfying $0 \le j < g$, provided the initial values $\mu_0(g)$ (contained in the top row of \cref{table:mu_j_unique_solution}) are known for all~$g$.
The required initial values $\mu_0(g) = |W_0(g)|$ were given by Bizley \cite{biz} in 1954. We shall express these values in \cref{thm:bizley} in terms of a quantity~$H_g$, introduced below. A derivation of $\mu_j(g)$ using only the values of $\mu_0(g)$ is given in \cite[Chapter~4]{firoozi-thesis}.
\begin{table}[ht]
    \centering
    \begin{tabular}{|c|c|c|c|c| >{\columncolor{gray!30}} c|}
    \hline
    \rowcolor{red!60}
         $\mu_0(1)  $                    &   $\mu_0(2)$ &   $\mu_0(3)$ &   $\mu_0(4)$ & $\mu_0(5)$&  $\cdots$  \\ \hline
        \multicolumn{1}{c|}{}  &   \cellcolor{blue!30} $\mu_1(2)$ &   $\mu_1(3)$ &   $\mu_1(4)$ & $\mu_1(5)$ & $\cdots$  \\ \cline{2-6}
        \multicolumn{2}{c|}{}   &  \cellcolor{blue!30}  $\mu_2(3)$ &   $\mu_2(4)$ & $\mu_2(5)$ & $\cdots$  \\ \cline{3-6}
        \multicolumn{3}{c|}{}     &  \cellcolor{blue!30}  $\mu_3(4)$ & $\mu_3(5)$ & $\cdots$   \\ \cline{4-6}
        \multicolumn{4}{c|}{}      & \cellcolor{blue!30}  $\mu_4(5)$& $\cdots$   \\ \cline{5-6}
        \multicolumn{5}{c|}{}      & \cellcolor{blue!30} 
 $\cdots$  \\ \cline{6-6}
    \end{tabular}
    \caption{
    All values $\mu_j(g)$ can be derived using only the values in the top row of the table, but can be more easily derived by also using the values in the coloured diagonal.}
    \label{table:mu_j_unique_solution}
\end{table}

However, the values $\mu_j(g)$ can be derived more easily by additionally making use of the values $\mu_{g-1}(g)$
(contained in the coloured diagonal of \cref{table:mu_j_unique_solution}).
These additional values can be obtained using \cref{cor:blocks} from the values $|W_{g(a+b)-1}(g)|$ given by Bizley~\cite{biz}, and can be expressed in terms of another quantity~$E_g$.
This is the approach we shall use to establish the value of $\mu_j(g)$ in \cref{thm:path_counting_formula}.

We now define the quantities $H_g$ and $E_g$ as sums over all integer partitions of~$g$. Recall that a weakly increasing sequence of positive integers $\lambda$ whose entries sum to $g$ is a \textit{partition of $g$};  we write $\lambda \vdash g$ to indicate this.  Each entry of $\lambda$ is called a \textit{part}.  
We use the notation $\lambda = \langle 1^{m_1}2^{m_2} \cdots \rangle$ to mean that $\lambda$ has $m_i$ parts equal to~$i$, so $g = \sum_{i \geq 1} i m_i$.
For example, the partition $(1,1,2,3)$ of $7$ is also written $\langle 1^22^13^1 \rangle \vdash 7$.

    For $i>0$, let
    \begin{equation}\phantomsection\label{def:ici}
        {c_i} \coloneqq \frac{1}{i(a+b)}\binom{i(a+b)}{ia}.
    \end{equation}

    For a partition $\lambda=\langle 1^{m_1}2^{m_2} \cdots \rangle\vdash g$, let its \textit{length} be $\ell(\lambda) \coloneqq \sum_{i\ge 1} m_i$, and let
    \begin{align}
        c_\lambda &\coloneqq \prod_{i\ge 1} \frac{c_i^{m_i}}{m_i!}.
        \label{def:C_lambda}
    \end{align}
    Now let
        \begin{align}
        {H}_g &\coloneqq \sum_{\lambda \vdash g}  c_\lambda,
        \label{eq:Hg_def}\\
        {E}_g &\coloneqq \sum_{\lambda \vdash g} (-1)^{g-\ell(\lambda)} c_\lambda,
        \label{eq:Eg_def}
    \end{align}
and for convenience let
\begin{equation}\label{eq:E0H0}
   E_0 \coloneqq 1 \quad \mbox{and} \quad H_0 \coloneqq 1.
\end{equation}
\begin{remark}
\label{rem:H_and_E_vs_symmetric_functions}
    For each positive integer~$g$, the quantities $H_g$ and $E_g$ are in fact specializations of (one part) \textit{complete} and \textit{elementary} symmetric functions $h_g$ and~$e_g$, respectively. The standard relationship between the \textit{power sums} $p_i$ and the complete and elementary symmetric functions are 
    \begin{align}
        h_g = \sum_{\lambda \vdash g} \frac{p_\lambda}{z_\lambda}\quad \textnormal{ and }\quad e_g = \sum_{\lambda \vdash g} (-1)^{g-\ell(\lambda)}\frac{p_\lambda}{z_\lambda},
    \end{align}
    where for a partition $\lambda = \langle 1^{m_1} 2^{m_2} \cdots \rangle$ of~$g$, we write $z_\lambda := \prod_{i \ge 1} i^{m_i} m_i!$ and $p_\lambda \coloneqq \prod_{i \ge 1} p_i^{m_i}$.  Then the quantities $H_g$ and $E_g$ in \eqref{eq:Hg_def} and \eqref{eq:Eg_def} are obtained from the specialization given by $p_i = ic_i$. The identity
    \begin{align}
    \label{eq:symfuncrel}
        \sum_{i=0}^g (-1)^i E_i H_{g-i} = 0
    \end{align}
    results from another well-known relationship between the complete and elementary symmetric functions.
    See \cite{aigner} or the classic reference \cite{macdonald} for details on the above and other relevant background on symmetric functions.
    \lrcornerqed
    \end{remark}

\begin{thm}[Bizley \cite{biz}]
\phantomsection\label{thm:actual_bizley}
    We have that
    \begin{align}
        |W_0(g)| &= {H}_g, \label{eq:biz:weak}\\
        |W_{g(a+b)-1}(g)| &= (-1)^{g+1}{E}_g. 
    \end{align}
\end{thm}
Using \cref{cor:blocks}, we then obtain the value of $\mu_0(g)$ and of~$\mu_{g-1}(g)$.

\begin{cor}[Value of $\mu_0(g)$ and $\mu_{g-1}(g)$]
\phantomsection
\label{thm:bizley}
    We have that 
    \begin{align}
        \mu_0(g) &= {H}_g,\label{eq:biz_Hg}\\
        \mu_{g-1}(g) &= (-1)^{g+1}{E}_g.\label{eq:biz_Eg}
    \end{align}
\end{cor}

We can now give a closed form expression for the value of~$\mu_j(g)$, which we remark is a truncated version of the sum from~\eqref{eq:symfuncrel}. 
This is our main enumerative result, and is the central consequence of \cref{thm:set_sizes}.

\begin{thm}[Path enumeration formula]
\phantomsection
\label{thm:path_counting_formula}
    We have
    \begin{equation}
         \mu_j(g) = \sum_{k=0}^j (-1)^k
        {E}_k {H}_{g-k}
        \quad \mbox{for $0 \le j < g$.}
    \end{equation}
\end{thm}
\begin{proof}
    The proof is by induction on~$j$. The base case
    $j=0$ is given by \eqref{eq:E0H0} and~\eqref{eq:biz_Hg}.
    Assume that the formula holds for all cases up to $j-1$, where $0 < j < g$.
    It follows from \cref{cor:rec} and \eqref{eq:biz_Hg} and \eqref{eq:biz_Eg} that
        \[
        \mu_j(g) = \mu_{j-1}(g) + (-1)^{j}{E}_{j}{H}_{g-j}.
        \]
    The inductive hypothesis then gives
    \begin{align}
        \mu_j(g)
        &= \sum_{k=0}^{j-1} (-1)^{k}{E}_{k}{H}_{g-k} + (-1)^{j}{E}_{j}{H}_{g-j}\\
        &= \sum_{k=0}^j (-1)^{k}{E}_{k}{H}_{g-k}, \label{eq:hook}
    \end{align}
    so case $j$ holds and the induction is complete.

\end{proof}

\begin{remark}
    The expression \eqref{eq:hook} can be written more compactly as a specialization of a Schur function via the identity $\sum_{k=0}^j (-1)^k e_k h_{g-k} = (-1)^j s_{\langle 1^j (g-j)^1 \rangle}$, where $s_{\langle 1^j (g-j)^1 \rangle}$ is the Schur function indexed by the hook partition $\langle 1^j (g-j)^1 \rangle$.  See \cite[Section I.3, Example 9]{macdonald}.
    \lrcornerqed
\end{remark}

\subsection{Background on rational Catalan combinatorics}
\label{sec:ratCat}

In the case $a=b=1$, the set $W_0(g)$ comprises all lattice paths from $(0,0)$ to~$(g,g)$ having no flaws. Such paths are known as \textit{Dyck paths}, and a classical result states that they are counted by the $g^{\rm th}$ \textit{Catalan number}
\begin{align}\label{eqn:Catalan}
C_g := \frac{1}{g+1}\binom{2g}{g}.
\end{align}
The \textit{rational Catalan numbers} \cite{rational_assoc-arm} are given by 
\begin{equation}
\label{eqn:rationalCatalan}
C_{a,b} := \frac{1}{a+b}\binom{a+b}{a},
\end{equation}
where $a,b$ are restricted to be coprime so that $C_{a,b}$ is always an integer. The quantity $C_{a,b}$ is equal to $c_1$ as defined in~\eqref{def:ici}. 
The specialization $C_{g,g+1}$ equals the Catalan number~$C_g$.
For a positive integer~$b$, the specialization $C_{g,gb+1}$ gives the 
\textit{Fuss--Catalan number}
\begin{align}
\label{eqn:fuss_catalan}
F_g^b \coloneq\frac{1}{gb+1}\binom{(b+1)g}{g}.    
\end{align}

Dyck paths are one of the many combinatorial objects counted by the Catalan numbers; 
at least 213 other combinatorial objects are also counted by these numbers~\cite{stanley-cat}, including binary trees and noncrossing partitions. 
The Fuss--Catalan numbers count common generalizations of Catalan objects~\cite{fuss}, including lattice paths from $(0,0)$ to $(g,gb)$ having no flaws with respect to the linear boundary of integer slope~$b$ (see \cite{cycle_lemma_og} for an argument expressed using  ballot sequences, or \cite{goulden_serrano} for a combinatorial proof using a construction akin to the \textit{reflection principle});
$b$-ary trees, which are of much interest in computer science (see for example~\cite{banderier} and its references); and certain generalizations of noncrossing partitions. 

The term \textit{rational Catalan combinatorics} is traditionally used to describe the study of objects counted by the rational Catalan numbers $C_{a,b}$ for general coprime~$a,b$. 
This includes the enumeration of \textit{rational Dyck paths}, namely lattice paths from $(0,0)$ to $(a,b)$ having no flaws with respect to the linear boundary of rational slope~$b/a$; 
rational associahedra and noncrossing matchings~\cite{rational_assoc-arm}; and quantities associated with rational parking functions~\cite{rat-armstrong}.  Fundamental statistics such as the area and dinv statistics on Dyck paths have also been generalized to the rational case~\cite{rat-vazirani}.  
Rational Dyck paths have been connected to many different objects through a new framework involving the \textit{signature} of coprime integers $a,b$~\cite{sig-ceb}.

Lattice paths from $(0,0)$ to $(ga,gb)$ having no flaws can be considerably more difficult to count in the case $g>1$ (when the co-ordinates of the endpoint are not coprime) than in the case $g=1$ (when they are coprime): for example, the expression $H_g$ for
$|W_0(g)|$ in \cref{thm:actual_bizley} is complicated to evaluate
in the case $g>1$, but simplifies to $C_{a,b}$ in the case $g=1$.
Though such paths are not necessarily counted by rational Catalan numbers, some authors consider these more general path-counting problems to be part of rational  Catalan combinatorics whereas others use the term \textit{rectangular Catalan combinatorics} \cite{open-berg}.
Similar complications in moving from the case $g=1$ to the case $g>1$ occur elsewhere, for example when relating results between rational Dyck paths and simultaneous core partitions (see \cite{anderson} and \cite{rat_non_coprime-vaz}), and 
in the study of parking functions and triangular partitions~\cite{park-berg, triangular-berg} involving a substantial generalization of Bizley's \cref{thm:actual_bizley} to symmetric functions.

\subsection{Counting paths that cross a linear boundary}

\label{sec:linbound}

The novelty of our work is in the combination of rational Catalan combinatorics with  the counting of paths that contain flaws.
To give further context, we briefly review the literature on the enumeration of lattice paths that cross a linear boundary.

The path enumeration setting we consider involves paths with step set $\{(1,0), (0,1)\}$ in the two-dimensional lattice~$\mathbb{Z}^2$; a boundary line joining the startpoint $(0,0)$ of a path to its endpoint~$(ga,gb)$, where $g,a,b$ are positive integers such that $a$ and $b$ are coprime; and $k$ flaws. 
Previous authors have defined a flaw differently from us, namely as a certain type of \underline{step} (usually a $(0,1)$ step) of the path that lies above the boundary.
Each of the references \cite{chen_CF_revisited,CF_cyc_shift_bdy,huq,tirrell, lattice_4_steps_flaws} adopts this step-based definition of flaw and a notion of the ``wrong'' side of the boundary, although the precise definition is not identical in all five references.

In the more general case that we consider here, where the value of $b/a$ need not necessarily be an integer, the definition of a flaw as a step is no longer appropriate since some steps can lie only partially above the boundary (see \cref{fig:step_issue}).
Our definition of a flaw, as a lattice \underline{point} of the path that lies above the boundary, does not have this ambiguity.

Note that these two definitions of flaws are genuinely different: 
\cref{fig:flaw_nonequivalent_definitions} shows that even in the case $a=b=1$ there is no simple relationship between the number of $(0, 1)$ steps lying above the boundary and the number of lattice points lying above the boundary.
In \cref{fig:all_six_lattice_paths} we further illustrate the difference between these two definitions of flaws when $g=2$ and $a=b=1$, which is the smallest non-trivial case of \cref{cor:blocks}. 

\begin{figure}[ht]
\definecolor{BurntOrange}{RGB}{235, 119, 52}
\definecolor{vertical_flaws}
{RGB}{235, 119, 52}
\centering
    \begin{tikzpicture}[scale = 0.7]
    \draw[gray!30] (0,0) grid (8,6);
    \draw[thick,red!60] (0,0) -- (8,6);
    \draw[ultra thick, blue] (0,0)  -- (1,0);
    \draw[ultra thick, blue] (1,2) -- (2,2);
    \draw[ultra thick, blue] (2,3) -- (3,3) -- (4,3);
    \draw[ultra thick, blue] (4,4) -- (5,4);
    \draw[ultra thick, blue] (5,4) -- (6,4);
    \draw[ultra thick, green] (6,4) -- (6,5);
    \draw[ultra thick, blue] (6,5) -- (7,5) -- (8,5) -- (8,6);
    \draw[ultra thick, green] (1,0) -- (1,1);
    \draw[ultra thick, vertical_flaws] (1,1) -- (1,2);
    \draw[ultra thick, vertical_flaws] (2,2) -- (2,3);
    \draw[ultra thick, vertical_flaws] (4,3) -- (4,4);
    \fill[red]	    (0,0) circle (1.5pt)
    			(4,3) circle (1.5pt)
    			(8,6) circle (1.5pt);
    \fill[orange]	(1,1) circle (3pt)
    			(1,2) circle (3pt)
                    (2,2) circle (3pt)
    			(5,4) circle (3pt)
                    (2,3) circle (3pt)
    			(3,3) circle (3pt)
                    (6,5) circle (3pt)
                    (4,4) circle (3pt);
    \end{tikzpicture}
\caption{A path with a boundary whose slope is not an integer. This shows that both $(0,1)$ steps and $(1,0)$ steps can lie partially above and below the boundary simultaneously, whereas lattice points cannot.}
\label{fig:step_issue}
\end{figure}
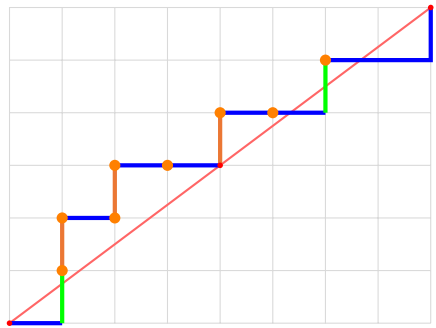

\begin{figure}[ht]
\centering
\begin{subfigure}{.4\textwidth}
  \centering
    \begin{tikzpicture}[scale = 0.8]
    \tikzmath{
    int \a, \b, \g;
    \a = 1;
    \b = 1;
    \g = 4;
    }
    \draw[gray!30] (0,0) grid (\g * \a ,\g * \b);
    \draw[thick,red!60] (0,0) -- (\g * \a ,\g * \b);
    \draw[ultra thick, blue] (0,0)  -- (1,0) -- (1,1);
    \draw[ultra thick, orange] (1,1) -- (1,2) -- (1,3);
    \draw[ultra thick, blue] (1,3) -- (2,3) -- (3,3) -- (4,3) -- (4,4);
    \foreach \i in {0,...,\g}
        \fill[red] (\a * \i,\b * \i) circle (1.5pt);
    \fill[orange]	    
        (1,2) circle (3pt)
        (1,3) circle (3pt)
        (2,3) circle (3pt);
    \end{tikzpicture}
    \caption{A path in which two $(0,1)$ steps and three lattice points lie above the boundary.}
  \label{fig:flaw_nonequivalent_definitions1}
\end{subfigure}%
\qquad 
\begin{subfigure}{.4\textwidth}
\centering
\begin{tikzpicture}[scale = 0.8]
    \tikzmath{
    int \a, \b, \g;
    \a = 1;
    \b = 1;
    \g = 4;
    }
    \draw[gray!30] (0,0) grid (\g * \a ,\g * \b);
    \draw[thick,red!60] (0,0) -- (\g * \a ,\g * \b);
    \draw[ultra thick, blue] (0,0)  -- (1,0) -- (1,1);
    \draw[ultra thick, orange] (1,1) -- (1,2);
    \draw[ultra thick, blue] (1,2) -- (2,2);
    \draw[ultra thick, orange] (2,2) -- (2,3);
    \draw[ultra thick, blue] (2,3) -- (3,3) -- (4,3) -- (4,4);
    \foreach \i in {0,...,\g}
        \fill[red] (\a * \i,\b * \i) circle (1.5pt);
    \fill[orange]	    
        (1,2) circle (3pt)
        (2,3) circle (3pt);
    \end{tikzpicture}
    \caption{A path in which two $(0,1)$ steps and two  lattice points lie above the boundary.}
\label{fig:flaw_nonequivalent_definitions2}
\end{subfigure}%
\caption{Even in the case  $a=b= 1$, there is no simple relationship between the number of $(0,1)$ steps lying above the boundary and the number of lattice points lying above the boundary.}
  \label{fig:flaw_nonequivalent_definitions}
\end{figure}
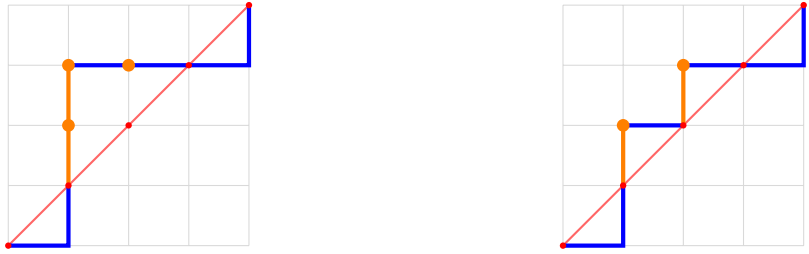 

\begin{figure}[ht]
\centering
    \tikzmath{
    int \a, \b, \g;
    \a = 1;
    \b = 1;
    \g = 2;
    }
    \begin{subfigure}{0.12\textwidth}
      \centering
        \begin{tikzpicture}[scale = 0.8]
        \draw[gray!30] (0,0) grid (\g*\a,\g*\b);
        \draw[thick,red!60] (0,0) -- (\g*\a,\g*\b);
        \draw[ultra thick, blue] (0,0) -- (1,0) -- (2,0) -- (2,1) -- (2,2);
        \foreach \i in {0,...,\g}
            \fill[red] (\a*\i,\b*\i) circle (1.5pt);
        \end{tikzpicture}
      \label{fig:patha}
    \end{subfigure}
\quad
    \begin{subfigure}{0.12\textwidth}
      \centering
        \begin{tikzpicture}[scale = 0.8]
        \draw[gray!30] (0,0) grid (\g*\a,\g*\b);
        \draw[thick,red!60] (0,0) -- (\g*\a,\g*\b);
        \draw[ultra thick, blue] (0,0) -- (1,0) -- (1,1) -- (2,1) -- (2,2);
        \foreach \i in {0,...,\g}
            \fill[red] (\a*\i,\b*\i) circle (1.5pt);
        \end{tikzpicture}
      \label{fig:pathb}
    \end{subfigure}
\quad
    \begin{subfigure}{0.12\textwidth}
      \centering
        \begin{tikzpicture}[scale = 0.8]
        \draw[gray!30] (0,0) grid (\g*\a,\g*\b);
        \draw[thick,red!60] (0,0) -- (\g*\a,\g*\b);
        \draw[ultra thick, blue] (0,0) -- (1,0) -- (1,1) -- (1,2) -- (2,2);
        \foreach \i in {0,...,\g}
            \fill[red] (\a*\i,\b*\i) circle (1.5pt);
        \fill[orange]	    
            (1,2) circle (3pt);
        \end{tikzpicture}
      \label{fig:pathc}
    \end{subfigure}
\quad
    \begin{subfigure}{0.12\textwidth}
      \centering
        \begin{tikzpicture}[scale = 0.8]
        \draw[gray!30] (0,0) grid (\g*\a,\g*\b);
        \draw[thick,red!60] (0,0) -- (\g*\a,\g*\b);
        \draw[ultra thick, blue] (0,0) -- (0,1) -- (1,1) -- (2,1) -- (2,2);
        \foreach \i in {0,...,\g}
            \fill[red] (\a*\i,\b*\i) circle (1.5pt);
        \fill[orange]	    
            (0,1) circle (3pt);
        \end{tikzpicture}
      \label{fig:pathd}
    \end{subfigure}
    \quad
    \begin{subfigure}{0.12\textwidth}
      \centering
        \begin{tikzpicture}[scale = 0.8]
        \draw[gray!30] (0,0) grid (\g*\a,\g*\b);
        \draw[thick,red!60] (0,0) -- (\g*\a,\g*\b);
        \draw[ultra thick, blue] (0,0) -- (0,1) -- (1,1) -- (1,2) -- (2,2);
        \foreach \i in {0,...,\g}
            \fill[red] (\a*\i,\b*\i) circle (1.5pt);
        \fill[orange]	    
            (0,1) circle (3pt)
            (1,2) circle (3pt);
        \end{tikzpicture}
      \label{fig:pathe}
    \end{subfigure}
 \quad
    \begin{subfigure}{0.12\textwidth}
      \centering
        \begin{tikzpicture}[scale = 0.8]
        \draw[gray!30] (0,0) grid (\g*\a,\g*\b);
        \draw[thick,red!60] (0,0) -- (\g*\a,\g*\b);
        \draw[ultra thick, blue] (0,0) -- (0,1) -- (0,2) -- (1,2) -- (2,2);
        \foreach \i in {0,...,\g}
            \fill[red] (\a*\i,\b*\i) circle (1.5pt);
        \fill[orange]	    
            (0,1) circle (3pt)
            (0,2) circle (3pt)
            (1,2) circle (3pt);
        \end{tikzpicture}
      \label{fig:pathf}
    \end{subfigure}
\caption{Set $a=b=1$ and $g=2$.  The six lattice paths from $(0,0)$ to $(2,2)$ are distributed as $(0,0,1,1,2,3)$ according to the number of lattice points above the boundary.  Here $W_0(2) = W_1(2) = \mu_0(2) = 2$ and $W_2(2)= W_3(2) = \mu_1(2) = 1$ (illustrating Corollaries \ref{cor:blocks} and~\ref{cor:decreasing}). 
The six lattice paths are distributed as $(0,0,1,1,2,2)$ according to the number of vertical steps above the boundary (illustrating the Chung--Feller result given in \cref{thm:CF}).}
\label{fig:all_six_lattice_paths}
\end{figure}

We review previous results relating to these two definitions of flaws in Sections~\ref{sec:first_setting} and~\ref{sec:second_setting}.

\subsubsection{Boundaries of integer slope with (0,1) steps as flaws}
\label{sec:first_setting}

In this part of the review, we take flaws to be $(0,1)$ steps of a path from $(0,0)$ to~$(g,gb)$ that lie above the boundary of integer slope~$b$.

Firstly consider paths from $(0,0)$ to~$(g,g)$. The number of such paths having $k=0$ flaws (no $(0,1)$ steps lying above the boundary, which is equivalent to having no lattice points above the boundary) is given by the Catalan number $C_g$ defined in~\eqref{eqn:Catalan}.
Chung and Feller's influential 1949 work~\cite{CF} showed that, remarkably, the same count applies for all~$k$.

\begin{thm}[Chung--Feller {\cite[Theorem 2A]{CF}}]\label{thm:CF}
    Let $k$ satisfy $0 \le k \le g$. Then the number of paths from $(0,0)$ to $(g,g)$ having $k$ of the $(0,1)$ steps lying above the boundary is~$C_g$.
\end{thm}

\cref{thm:CF} can be proven using bijective methods~\cite{callan1995pair}. 
Huq generalized \cref{thm:CF} to paths from $(0,0)$ to~$(g,gb)$ for each positive integer $b$.

\begin{thm}[Huq {\cite[Corollary 5.1.2]{huq}}]\label{thm:huq}
    Let $k$ satisfy $0 \le k \le gb$. Then the number of paths from $(0,0)$ to $(g,gb)$ having $k$ of the $(0,1)$ steps lying above the boundary is
    $F_g^b$, the Fuss--Catalan number 
defined in~\eqref{eqn:fuss_catalan}.
\end{thm}

Further variations on \cref{thm:CF} have been found \cite{CF_cyc_shift_bdy, CF_gen,tirrell}.

\subsubsection{Boundaries of rational slope with lattice points as flaws}
\label{sec:second_setting}

In this part of the review, we take flaws to be lattice points of a path from $(0,0)$ to~$(ga,gb)$ that lie strictly above the boundary. The number of these flaws is the measure $k$ used for $W_k(g)$ in \cref{def:path_sets}.

In 1950, Grossman~\cite{grossman1950paths} conjectured an explicit formula for the number of paths from $(0,0)$ to $(ga,gb)$ 
that may touch, but never rise above, the boundary $ay = bx$.
In our terminology, such paths have no flaws and so are counted by $|W_0(g)|$.
In~1954, Bizley \cite[Eq.~(10)]{biz} proved Grossman's formula using generating functions. 
Bizley \cite[Eq.~(8)]{biz} also obtained an explicit formula for the number of paths 
that lie wholly below, and do not touch, the boundary $ay=bx$ at an interior boundary point.
In other words, such paths have neither flaws nor interior boundary points.
Since the set of such paths is in bijection with the set of paths having max flaws (via rotation), this second result of Bizley's gives the value $|W_{g(a+b)-1}(g)|$. The values $|W_0(g)|$ and $|W_{g(a+b)-1}(g)|$ are stated in \cref{thm:actual_bizley}. 

One of the methods used by Bizley to prove \cref{thm:actual_bizley} involves cyclic rotation of paths. Such methods, in particular the \textit{Cycle Lemma}, have been frequently used to count lattice paths and have been popularized by many authors (see for example \cite{cycle_lemma_applications,cycle_lemma_og}).  Although we were able to use arguments similar to the Cycle Lemma to obtain results in certain special cases, we were not able to use it to obtain our general result.

\section{Special cases and example of the main enumerative result}
In this section we examine various special cases of the path enumeration formula of \cref{thm:path_counting_formula}, and give an example of its application. 

We firstly state the special case $g=1$.

\begin{thm}[Special case $g=1$]
\label{thm:fir_mar_rat}
We have
    \begin{equation}
        |W_k(1)| = \frac{1}{a+b}\binom{a+b}{a}
       \quad
         \mbox{for all $k$ satisfying $0 \le k < a+b$}.
    \end{equation}
\end{thm}
\begin{proof}
Let $k$ satisfy $0 \le k < a+b$.
By \cref{cor:blocks,thm:bizley} with $g=1$, we have
\begin{align}
|W_k(1)| = \mu_0(1) =  H_1 \label{eqn:H1}.
\end{align}
Now by \eqref{eq:Hg_def} and \eqref{def:C_lambda} we have
\[
H_1 = c_{\langle 1^1 \rangle} = c_1.
\]
Substitute into \eqref{eqn:H1} and use \eqref{def:ici} to give the result.
\end{proof}

\cref{thm:fir_mar_rat} shows that $|W_k(1)|$ is independent of~$k$ when $g=1$.
We highlight this special case because our general construction used to prove \cref{thm:set_sizes} simplifies greatly in the case $g=1$.  We re-examine this special case in \cref{sec:geqaul1}.

We next give the special case $g=2$.

\begin{thm}[Special case $g=2$]\label{thm:speccase2}
We have
\begin{equation}
|W_k(2)| = \begin{cases}  
    \displaystyle{\frac{1}{2(a+b)}\left[ \binom{2a+2b}{2a}+\frac{1}{a+b}\binom{a+b}{a}^2 \right]} &
     \mbox{for all $k$ satisfying $0 \le k < a+b$},\label{eq:bizg2} \\[3ex]
    \displaystyle{\frac{1}{2(a+b)}\left[ \binom{2a+2b}{2a}-\frac{1}{a+b}\binom{a+b}{a}^2 \right]} &
    \mbox{for all $k$ satisfying $a+b \le k < 2(a+b)$}.
    \end{cases}
\end{equation}
\end{thm}
\begin{proof}
By \cref{cor:blocks} with $g=2$,
\begin{align}
|W_k(2)| &= \begin{cases}  
     \mu_0(2) & \mbox{for $0 \le k < a+b$}, \\
     \mu_1(2) & \mbox{for $a+b \le k < 2(a+b),$}
    \end{cases}\\
        &= \begin{cases}  
     H_2 & \mbox{for $0 \le k < a+b$}, \\
     -E_2 & \mbox{for $a+b \le k < 2(a+b)$} 
     \label{eqn:H2E2}
    \end{cases}
\end{align}
using \cref{thm:bizley}.
Now by \eqref{eq:Hg_def} and \eqref{def:C_lambda} we have
\[
H_2 = c_{\langle 2^1 \rangle} + c_{\langle 1^2 \rangle} = c_2 + \tfrac{1}{2}c_1^2,
\]
and by \eqref{eq:Eg_def} and \eqref{def:C_lambda} we have
\[
-E_2 = c_{\langle 2^1 \rangle} - c_{\langle 1^2 \rangle} = c_2 - \tfrac{1}{2}c_1^2.
\]
Substitute into \eqref{eqn:H2E2} and use \eqref{def:ici} to give the result.
\end{proof}

We now consider the special case $a=1$, when the slope of the boundary is the positive integer~$b$ and the startpoint and endpoint are $(0,0)$ and $(g,gb)$, respectively. Although \cref{thm:path_counting_formula} already provides an expression for $\mu_j(g)$ in this case, we now derive an alternative formula involving Fuss--Catalan numbers that appears to be simpler.

\begin{thm}[Alternative formula for $a=1$]
    \label{thm:path_counting_aisone_case}
        Let $a=1$. Then for all $j$ satisfying $0 \leq j < g$,
           \begin{equation}
		\label{eqn:a=1}
                |W_{j(b+1)}(g)| = |W_{j(b+1)+1}(g)| = \cdots = |W_{j(b+1)+b}(g)| = \mu_j(g) = \sum_{\myatop{i_1 + \cdots + i_{b+1} = g-1}{i_{b+1} \leq g-1-j}} F_{i_1}^b  \cdots F_{i_{b+1}}^b
            \end{equation}
where $F_i^b$ is the Fuss--Catalan number 
defined in~\eqref{eqn:fuss_catalan}.
\end{thm}
   
\begin{proof}
Fix the positive integer~$g$. All equalities of \eqref{eqn:a=1} except the last hold by \cref{cor:blocks}. 
We shall establish the last equality by induction on~$j$.
We shall describe the structure of the induction by reference to \cref{table:mu_j_unique_solution}, numbering table rows from $0$ and columns from~$1$.

Let $P_j$ be the statement that 
\[
\mu_j(t) = \sum_{\myatop{i_1 + \cdots + i_{b+1} = t-1}{i_{b+1} \leq t-1-j}} F_{i_1}^b  \cdots F_{i_{b+1}}^b \quad \mbox{for all $t$ satisfying $j+1 \le t \le g$}.
\]
In other words, $P_j$ is the statement that 
 the last equality of \eqref{eqn:a=1} holds for all entries in row $j$ of \cref{table:mu_j_unique_solution} up to and including column~$g$.
We prove the theorem by showing by induction on $j$ that $P_j$ holds for $0 \le j < g$: that is, for all entries in rows $0, 1, \dots, g-1$ of \cref{table:mu_j_unique_solution} up to and including column~$g$. 

For the base case $P_0$, let $t$ satisfy $1 \le t \le g$. We require that
\[
\mu_0(t) = \sum_{i_1 + \cdots + i_{b+1} = t-1} F_{i_1}^b  \cdots F_{i_{b+1}}^b. 
\]
This is equivalent to showing that $\mu_0(t) = F_t^b$, using the identity $F_t^b = \sum_{i_1 + \cdots + i_{b+1} = t-1} F_{i_1}^b \cdots F_{i_{b+1}}^b$;
this identity  can be obtained by applying Lagrange inversion \cite[Theorem 1.2.4]{goul} to the functional equation $F(x) = xF(x)^{b+1} + 1$ satisfied by the generating series $F(x) = \sum_{i\geq0} F_i^b x^i$  (see \cite[p.~362]{graham1994concrete} for the identity and functional equation). 
Since $a=1$, the set $W(t)$ comprises all paths from $(0,0)$ to~$(t,tb)$. As discussed in Section \ref{sec:ratCat},
the number of such paths with zero flaws is~$F_t^b$ and so $\mu_0(t) = F_t^b$ as required. This establishes the base case~$P_0$. 

Now let $j$ satisfy $0 < j < g$ and assume that the case $P_{j-1}$ holds,
so that
\begin{equation}
\label{muj-1t}
\mu_{j-1}(t) = \sum_{\myatop{i_1 + \cdots + i_{b+1} = t-1}{i_{b+1} \leq t-j}} F_{i_1}^b \cdots F_{i_{b+1}}^b \quad \mbox{for all $t$ satisfying $j \le t \le g$}.
\end{equation}
Let $t$ satisfy $j+1 \le t \le g$. Then the recurrence relation of~\cref{cor:rec} gives
\begin{equation}
    \mu_j(t) = \mu_{j-1}(t) - \mu_0(t-j) \, \mu_{j-1}(j) =  \mu_{j-1}(t) - F_{t-j}^b \, \mu_{j-1}(j). \label{eq:mujgnew} 
\end{equation}
Apply the inductive hypothesis~\eqref{muj-1t} to $\mu_{j-1}(t)$ and $\mu_{j-1}(j)$ to deduce that
\begin{align}
    \mu_j(t) &= \sum_{\myatop{i_1 + \cdots + i_{b+1} = t-1}{i_{b+1} \leq t - j}} F_{i_1}^b \cdots F_{i_{b+1}}^b  - F_{t-j}^b \sum_{\myatop{i_1 + \cdots + i_{b+1} = j-1}{i_{b+1} = 0}} F_{i_1}^b \cdots F_{i_{b+1}}^b \\
    &= \sum_{\myatop{i_1 + \cdots + i_{b+1} = t-1}{i_{b+1} \leq t - j}} F_{i_1}^b \cdots F_{i_{b+1}}^b  - \sum_{i_1 + \cdots + i_{b} = j-1} F_{i_1}^b \cdots F_{i_{b}}^b F_{t-j}^b
\end{align}
using that $F_0^b=1$. This gives
\[
    \mu_j(t) = \sum_{\myatop{i_1 + \cdots + i_{b+1} = t-1}{i_{b+1} \leq t-1-j}} F_{i_1}^b \cdots F_{i_{b+1}}^b, 
\]
by expanding the right hand side according to whether $i_{b+1} \le t-1-j$ or $i_{b+1} = t-j$. Therefore the case $P_j$ holds, completing the induction.
\end{proof}

We next give an alternative formula for the special case $a=b=1$ (when the slope of the 
boundary is~$1$) by taking $b=1$ in Theorem~\ref{thm:path_counting_aisone_case}
and noting that the Fuss--Catalan number $F_i^1$ equals the Catalan number~$C_i$,
in order to demonstrate a further simplification. 

\begin{cor}[Alternative formula for $a=b=1$]
    \label{cor:path_counting_aisbisone_case}
        Let $a=b=1$. Then
           \begin{equation}
                |W_{2j}(g)| = |W_{2j+1}(g)| = \mu_j(g) = \sum_{k = j}^{g-1} C_{k}  C_{g-1-k}
                               \quad\mbox{for all $j$ satisfying $0\le j< g$,}
            \end{equation}
                
where $C_i$ is the 
Catalan number defined in \eqref{eqn:Catalan}.
\end{cor}

The reader is invited to compare 
Theorem~\ref{thm:path_counting_aisone_case} with Theorem~\ref{thm:huq}: 
both apply to a boundary of integer slope~$b$, but 
Theorem~\ref{thm:path_counting_aisone_case}
takes flaws to be points above the boundary whereas 
Theorem~\ref{thm:huq}
takes flaws to be $(0,1)$ steps above the boundary.
The count in Theorem~\ref{thm:path_counting_aisone_case} depends on the number of flaws whereas the count in Theorem~\ref{thm:huq} does not.

We now give an example of how to apply \cref{thm:path_counting_formula}.

\begin{example}[Computation using the path enumeration formula]
Let $(a,b) = (3,2)$ and $g=4$. We illustrate the use of the path enumeration formula  \cref{thm:path_counting_formula} to calculate the number $|W_k(4)|$ of paths from $(0,0)$ to $(12,8)$ having $k$ flaws, for each $k$ satisfying $0 \le k < 20$.
By \cref{cor:blocks}, it is sufficient to determine $\mu_j(4)$ for each $j = 0,1,2,3$.

    We begin by listing the  partitions of the integers $1,2,3,4$.
    \begin{align}
        \mbox{Partitions of }4 &: \langle 4^1 \rangle, \ \langle 1^1 3^1\rangle, \ \langle 2^2 \rangle, \ \langle 1^2 2^1 \rangle, \ \langle 1^4 \rangle, \ \\
        \mbox{Partitions of }3 &: \langle 3^1 \rangle, \ \langle 1^1 2^1\rangle, \ \langle 1^3 \rangle, \ \\
        \mbox{Partitions of }2 &:\langle 2^1 \rangle, \ \langle 1^2 \rangle, \ \\
        \mbox{Partitions of }1 &:\langle1^1 \rangle.
    \end{align} 
    Using \eqref{def:ici}, we compute  
    $$
    c_1 = 2, \ c_2 = 21,\ c_3 = \frac{1001}{3}, \ c_4 = \frac{12597}2.
    $$
    Using \eqref{def:C_lambda}, we then compute (for example)
    $$
    c_{\langle 1^2 2^1\rangle} = \left( \frac{c_1^2}{2!}\right) \left( \frac{c_2^1}{1!}\right) = 42, \quad
    c_{\langle 1^3\rangle} = \left( \frac{c_1^3}{3!}\right) = \frac{4}{3}.
    $$
    The full set of $c_\lambda$ values for $\lambda \vdash j$ where $1 \leq j \leq 4$ is
    \begin{alignat}{5}
        & c_{\langle 4^1 \rangle}=\frac{12597}2, \quad && c_{\langle 1^1 3^1\rangle}=\frac{2002}{3}, \quad && c_{\langle 2^2 \rangle}=\frac{441}{2}, \quad && c_{\langle 1^22^1 \rangle}=42, \quad && c_{\langle 1^4\rangle}= \frac23,\\
        & c_{\langle 3^1 \rangle}=\frac{1001}{3}, \quad && c_{\langle 1^1 2^1\rangle}=42, \quad && c_{\langle 1^3 \rangle}=\frac43, \quad && &&\\
        & c_{\langle 2^1 \rangle}= 21, \quad && c_{\langle 1^2 \rangle}=2, \quad && && &&\\
        & c_{\langle1^1 \rangle}=2. \quad  && && && &&
    \end{alignat} 

    Using \eqref{eq:Hg_def} and \eqref{eq:Eg_def}, we next calculate (for example)
    \begin{align}
    {H}_3 &= 
    c_{\langle 3^1 \rangle}
    + c_{\langle 1^1 2^1\rangle}
    +c_{\langle 1^3 \rangle}
    =\frac{1001}{3}+42+\frac43 = 377,\\
    {E}_3 &= 
    (-1)^{3-1}c_{\langle 3^1 \rangle}
    + (-1)^{3-2}c_{\langle 1^1 2^1\rangle}
    +(-1)^{3-3}c_{\langle 1^3 \rangle}
    =\frac{1001}{3}-42+\frac43 = 293.
    \end{align} 
    The full set of ${H}_k$ and ${E}_k$ values is
    \begin{alignat}{2}
        {H}_4 &= 7229,  \qquad &&{E}_4 = -5452,\\
        {H}_3 &= 377, \qquad &&{E}_3 = 293,\\
        {H}_2 &= 23, \qquad &&{E}_2 =-19, \\
        {H}_1 &= 2, \qquad &&{E}_1 = 2, \\
        & &&{E}_0 = 1.
    \end{alignat}
    
    Using \cref{thm:path_counting_formula}, we then determine that 
    \begin{align}
        \mu_0(4) &= 
        {E}_0 {H}_{4} =1 \cdot 7229= 7229\\
        \mu_1(4) &= 
        {E}_0 {H}_{4}-
        {E}_1 {H}_{3} = 1 \cdot 7229 - 2 \cdot 377 = 6475\\
        \mu_2(4) &= 
        {E}_0 {H}_{4}-
        {E}_1 {H}_{3}+ 
        {E}_2 {H}_{2} =  1 \cdot 7229 - 2 \cdot 377- 19 \cdot 23 =6038\\
        \mu_3(4) &= 
        {E}_0 {H}_{4}-
        {E}_1 {H}_{3}+ 
        {E}_2 {H}_{2}-
        {E}_3 {H}_{1}=  1 \cdot 7229 - 2 \cdot 377- 19 \cdot 23- 293 \cdot 2 =5452.
    \end{align}
    (Alternatively, we may use \eqref{eq:biz_Eg} for a more direct calculation of the last value $\mu_3(4) = (-1)^{4+1} {E}_4 = 5452$.)
    
    Using \cref{cor:blocks}, we may now determine the value of $|W_k(4)|$ for each $k$ satisfying $0 \le k <20$. The resulting values agree with the computer enumeration shown in \cref{table:computer_enumeration}. \lrcornerqed
\end{example}

We remark, as noted by Bizley~\cite{biz}, that both $H_g$ and $E_g$ are necessarily integers because of the counting result \cref{thm:actual_bizley}, even though this is not readily apparent from the forms \eqref{def:C_lambda}, \eqref{eq:Hg_def}, and~\eqref{eq:Eg_def}.
We further remark that although the quantity $c_i$ defined in \eqref{def:ici} is not necessarily an integer, it is not difficult to show that $ic_i$ is an integer.

\section{Proof of the main combinatorial result}
\label{sec:bijection}

For convenience, we restate our main combinatorial result here.

{\bf \cref{thm:set_sizes}.}
\textit{Let $g,k$ satisfy $0 \le k < g(a+b)-1$. Then}
    $$
        |W_k(g) \setminus S_k(g)| = |W_{k+1}(g)|.
    $$

\subsection{Proof outline}

We shall prove our main combinatorial result by considering fixed $g,k$ satisfying $0 \le k < g(a+b)-1$ and constructing injective maps
\begin{align}
    \phi &\colon W_k(g) \setminus S_k(g) \to W_{k+1}(g),\\
    \psi &\colon W_{k+1}(g) \to W_k(g) \setminus S_k(g).
\end{align} 
In fact, the map $\psi$ we shall construct is the inverse of~$\phi$, although we shall not require this fact in our proof.
We partition the set $W_k(g) \setminus S_k(g)$ into subsets $X$ and~$Y$, and partition (using a different rule) the set $W_{k+1}(g)$ into subsets $\mathcal{X}$ and~$\mathcal{Y}$. We allow each of the partitioning subsets to be empty.
Using these partitions, we then specify the action of $\phi$ using injective submaps $\phi^X$ and~$\phi^Y$, and the action of $\psi$ using injective submaps $\psi^{\mathcal X}$ and~$\psi^{\mathcal Y}$
(see \cref{fig:phi_construction_new}). 

\begin{figure}[ht]
    \centering
    \begin{tikzpicture}[scale = 0.5]
    \pgfmathsetmacro{\cellHeight}{3}
    \pgfmathsetmacro{\cellWidth}{6}
    \pgfmathsetmacro{\SkGap}{1}
    \pgfmathsetmacro{\XYGap}{1}
    \pgfmathsetmacro{\domainCodomainGap}{4}
    \pgfmathsetmacro{\phiPsiGap}{0.4}
        \draw[black, thick] 
            (0,\cellHeight+\XYGap) rectangle (\cellWidth, 2*\cellHeight+\XYGap);
        \draw[black, thick] 
            (0,0) rectangle (\cellWidth,\cellHeight);
        \draw[black, thick, fill = gray!20] 
            (0,-\SkGap-\cellHeight) rectangle (\cellWidth,-\SkGap);

        \draw[black, thick] 
            (\cellWidth+\domainCodomainGap,\cellHeight+\XYGap) rectangle (2*\cellWidth+\domainCodomainGap,2*\cellHeight+\XYGap);
        \draw[black, thick] 
            (\cellWidth+\domainCodomainGap,0) rectangle (2*\cellWidth+\domainCodomainGap,\cellHeight);
        
        \node (X) at (0.5*\cellWidth,1.5*\cellHeight+\XYGap) {$X$};
        \node (Y) at (0.5*\cellWidth,0.5*\cellHeight) {$Y$};
        \node (S) at (0.5*\cellWidth,-\SkGap-0.5*\cellHeight) {$S_k(g)$};
        
        \node (XX) at (1.5*\cellWidth+\domainCodomainGap,1.5*\cellHeight+\XYGap) {$\mathcal{X}$};
        \node (YY) at (1.5*\cellWidth+\domainCodomainGap,0.5*\cellHeight) {$\mathcal{Y}$};
        
        \draw[->,black] 
            (0.5+\cellWidth,1.5*\cellHeight+\XYGap+0.5*\phiPsiGap) -- (-0.5+\cellWidth+\domainCodomainGap,1.5*\cellHeight+\XYGap+0.5*\phiPsiGap) node 
            [midway, sloped, above] {$\phi   ^X$};
        \draw[->,black] 
            (0.5+\cellWidth,0.5*\cellHeight+0.5*\phiPsiGap) -- (-0.5+\cellWidth+\domainCodomainGap,0.5*\cellHeight+0.5*\phiPsiGap) node 
            [midway, sloped, above] {$\phi   ^Y$};
        \draw[->,black] 
            (-0.5+\cellWidth+\domainCodomainGap,1.5*\cellHeight+\XYGap-0.5*\phiPsiGap) -- (0.5+\cellWidth,1.5*\cellHeight+\XYGap-0.5*\phiPsiGap) node 
            [midway, sloped, below] {$\psi   ^\mathcal X$};
        \draw[->,black] 
            (-0.5+\cellWidth+\domainCodomainGap,0.5*\cellHeight-0.5*\phiPsiGap) -- (0.5+\cellWidth,0.5*\cellHeight-0.5*\phiPsiGap) node 
            [midway, sloped, below] {$\psi   ^\mathcal Y$};

        \node[anchor = east] (Nk) at (-1,\cellHeight+0.5*\XYGap){$W_k(g)$};
        \node[anchor = west] (Nk1) at (1+2*\cellWidth+\domainCodomainGap,\cellHeight+0.5*\XYGap){$W_{k+1}(g)$};
        
        \pgfmathsetmacro{\aspectNumber}{(2*\cellHeight+\SkGap+0.5*\XYGap)/(3*\cellHeight+\SkGap+\XYGap)}
        \draw [thick, decorate,
            decoration = {brace, raise = 6pt, 
            aspect = \aspectNumber, amplitude = 8pt}] 
            (0,-\cellHeight-\SkGap) --  (0,2*\cellHeight+\XYGap);
        \draw [thick, decorate,
            decoration = {brace, mirror, 
            raise = 6pt, amplitude = 8pt}] 
                (2*\cellWidth+\domainCodomainGap,0) --  (2*\cellWidth+\domainCodomainGap,2*\cellHeight+\XYGap);
    \end{tikzpicture}
    \caption{%
        The map $\phi    \colon W_k(g) \setminus S_k(g) \to W_{k+1}(g)$ is defined piecewise using the maps $\phi^X: X \to \mathcal X$ and~$\phi^Y: Y \to \mathcal Y$. 
        The map $\psi    \colon W_{k+1}(g) \to W_k(g) \setminus S_k(g)$ is likewise defined piecewise using the maps $\psi^\mathcal X: \mathcal X \to X$ and~$\psi^\mathcal Y: \mathcal Y \to Y$.}
    \label{fig:phi_construction_new}
\end{figure}

To prove \cref{thm:set_sizes}, it suffices to 
\begin{enumerate}
    \item
        specify the partition of $W_k(g)\setminus S_k(g)$ and of $W_{k+1}(g)$ as illustrated in \cref{fig:phi_construction_new}, and 
    \item
        define the maps $\phi$ and~$\psi$, and show that they are both injective.
\end{enumerate}

\subsection{Partition of sets $W_k(g)\setminus S_k(g)$ and $W_{k+1}(g)$}
\label{subsec:partition}

We first introduce some additional terminology.  Recall that the boundary of a path in $W(g)$ is the line from $(0,0)$ to $(ga, gb)$.

\begin{definition}[Elevation]
\label{def:elev}
Let $(i,j)$ be a point of a path in $W(g)$.  The \textit{elevation} of $(i,j)$ is~${ja - ib}$.~\lrcornerqed
\end{definition}

The elevation of a point of a path in $W(g)$ is a measure of the directed distance from the point to the boundary. 
Points on the boundary have zero elevation; points above the boundary have positive elevation;  points below the boundary have negative elevation.

\begin{definition}[Lowest points above, highest points below]
\label{def:hpblpa}
    Let $p$ be a path.  
    The \textit{lowest points above the boundary (LPAs)} of $p$ are those points of $p$ (if any) 
    attaining the smallest strictly positive elevation.
    The \textit{highest points below the boundary (HPBs)} of $p$ are defined analogously.
\lrcornerqed
\end{definition}

See \cref{fig:subpath_vs_path_1} for an illustration of a path $p$ with LPAs $L,L',L''$ and HPBs $H,H'$. We note that the possible elevation values for an LPA are $1,2,\dots,\mathrm{min}(a,b)$, and that the possible elevation values for an HPB are $-1,-2,\dots,-\mathrm{min}(a,b)$.
Both the number of LPAs and the number of HPBs of a path in $W(g)$ lie in $\{0,1,\dots, g\}$.

We shall often consider a \textit{subpath} $p'$ of a path~$p$, namely a consecutive sequence of steps of~$p$. When viewed as a separate path in its own right, the boundary of $p'$ need not coincide with the boundary of~$p$ (nor even have the same slope) and so its LPAs and HPBs need not necessarily be the same as those of~$p$ (see \cref{fig:subpath_vs_path}). When we wish to view $p'$ as a path in its own right, we shall refer to ``the path~$p'$'';  when we wish to view $p'$ as a part of $p$ we shall refer to ``the subpath~$p'$''.

\cref{def:concatenation} describes the combination of paths $p_1$ and $p_2$ to form the concatenated path $p_1p_2$. To reverse this process, we \textit{split} the path $p_1p_2$ at the endpoint of $p_1$ into component paths $p_1, p_2$. We can similarly split a path at two distinct points to form component paths $p_1, p_2, p_3$. 
If the elevation of the startpoint and endpoint of $p_i$ (viewed as a subpath) are equal, then $p_i$ (viewed as a path) has a boundary with the same slope (the same values of $a$ and~$b$) as the full path.

We make the following key observation about the change in the number of flaws when a path is split at an HPB or LPA and the resulting subpaths are interchanged.

\begin{obs}\label{obs:flaws}
    Let $p$ be a path containing exactly $\beta$ interior boundary points and exactly $\lambda$ LPAs.
    Suppose that $p$ is split at an HPB $H$ into $p_1p_2$, so that $H$ is the endpoint of $p_1$ and the startpoint of~$p_2$. Then the rearranged path $p_2p_1$ has exactly $\beta+1$ more flaws than~$p$, namely all $\beta$ interior boundary points of $p$ together with the endpoint of~$p_2$.
    If instead $p$ is split at an LPA into $p_1p_2$, then the rearranged path $p_2p_1$ has exactly $\lambda$ fewer flaws than~$p$, namely all $\lambda$ LPAs of~$p$.
    \lrcornerqed
\end{obs}

See \cref{fig:cyc_perm_observation} for an illustration of \cref{obs:flaws}.

\begin{figure}[ht]
\begin{center}
    \tikzmath{
    int \a, \b, \g, \n, \m;
    \a = 3;
    \b = 2;
    \g = 3;
    \n = \g*\a;
    \m = \g*\b;
    }
\definecolor{brown2}{RGB}{0,0,0}
\definecolor{subpath_vs_path_HPB_color}{RGB}{252, 0, 181}
\begin{subfigure}{.4\textwidth}
    \begin{tikzpicture}[scale = 0.8]
        \draw[gray!30] (0,0) grid (\g * \a ,\g * \b);
        \draw[thick,red!50] (0,0) -- (\g * \a ,\g * \b);
        \draw[black,line width = 0.8pt, dotted]
            (2-2*2/13, 2+2*3/2*2/13) -- (2+6*2/13, 2-6*3/2*2/13);
        \draw[black,line width = 0.8pt, dotted]
            (8-2*2/13, 6+2*3/2*2/13) -- (8+6*2/13, 6-6*3/2*2/13);
        \draw[ultra thick, blue] 
            (0,0)  -- (1,0) -- (2,0) -- (2,1) -- (2,2);
        \draw[ultra thick, brown2]
            (2,2) -- (3,2) -- (4,2) -- (5,2) -- (5,3) -- (5,4) -- (5,5) -- (6,5) -- (6,6) -- (7,6) -- (8,6);
        \draw[ultra thick, blue] 
            (8,6) -- (9,6);
	\foreach \i in {0,...,\g}
		\fill[red] (\i*\a, \i*\b) circle (2pt);
        \fill[teal]	
            (2,2) circle (3pt)
            (5,4) circle (3pt)
            (8,6) circle (3pt);
        \fill[orange]  
            (3,2) circle (3pt);
        \fill[subpath_vs_path_HPB_color]
            (2,1) circle (3pt)
            (5,3) circle (3pt);
        \draw[shift = {(2.5,4.5)}, color = brown2] node[font=\Large] {$p'$};
        \draw[shift = {(1.5,2.5)}, color = teal] node[font=\Large] {$L$};
        \draw[shift = {(4.5,4.5)}, color = teal] node[font=\Large] {$L'$};
        \draw[shift = {(7.5,6.5)}, color = teal] node[font=\Large] {$L''$};
        \draw[shift = {(3.5,1.5)}, color = orange] node[font=\Large] {$P$};
        \draw[shift = {(2.5,0.5)}, color = subpath_vs_path_HPB_color] node[font=\Large] {$H$};
        \draw[shift = {(5.5,2.5)}, color = subpath_vs_path_HPB_color] node[font=\Large] {$H'$};
        
    \end{tikzpicture}
    \caption{A path~$p$ and (in black) its subpath~$p'$. \phantom{newline newline}}
    \label{fig:subpath_vs_path_1}
\end{subfigure}
\qquad
%
\begin{subfigure}{.4\textwidth}
    \begin{tikzpicture}[scale = 0.8]
        \draw[gray!10] (0,0) grid (\g * \a ,\g * \b);
        \draw[gray!30] (2,2) grid (\g * \a -1 ,\g * \b);
        \draw[thick,red!50] (2,2) -- (\g * \a -1 ,\g * \b);
        \draw[ultra thick, brown2]
            (2,2) -- (3,2) -- (4,2) -- (5,2) -- (5,3) -- (5,4) -- (5,5) -- (6,5) -- (6,6) -- (7,6) -- (8,6);
	\foreach \i in {0,...,\g}
		\phantom{\fill[red] (\i*\a, \i*\b) circle (2pt);}
        \fill[teal]	
            (2,2) circle (3pt)
            (5,4) circle (3pt)
            (8,6) circle (3pt);
        \fill[orange]  
            (3,2) circle (3pt);
        \fill[orange]  
            (6,5) circle (3pt);
        \fill[subpath_vs_path_HPB_color]
            (5,3) circle (3pt);
        \draw[shift = {(2.5,4.5)}, color = brown2] node[font=\Large] {$p'$};
        \draw[shift = {(1.5,2.5)}, color = teal] node[font=\Large] {$L$};
        \draw[shift = {(4.5,4.5)}, color = teal] node[font=\Large] {$L'$};
        \draw[shift = {(7.5,6.5)}, color = teal] node[font=\Large] {$L''$};
        \draw[shift = {(3.5,1.5)}, color = orange] node[font=\Large] {$P$};
        \draw[shift = {(5.5,5.5)}, color = orange] node[font=\Large] {$Q$};
        \draw[shift = {(5.5,2.5)}, color = subpath_vs_path_HPB_color] node[font=\Large] {$H'$};
    \end{tikzpicture}
    \caption{The subpath $p'$ as a path in its own right.}
    \label{fig:subpath_vs_path_2}
\end{subfigure}
\caption{Let $(a,b)=(3,2)$. 
The path $p$ is a member of $W_7(3) \setminus S_7(3)$ containing the interior boundary point $P$ and the LPAs $L, L',L''$ and HPBs $H,H'$. 
The subpath $p'$ contains the same interior boundary point and LPAs as~$p$, as well as the HPB $H'$ of~$p$.
The path $p'$ (on its own) has the same slope as~$p$, but 
is a member of $S_4(2) \subseteq W_4(2)$ containing the boundary points $L, L', L''$, the unique HPB~$P$, and the unique LPA~$Q$. We note that $H'$ is not an HPB of the path~$p'$.
}
\label{fig:subpath_vs_path}
\end{center}
\end{figure}
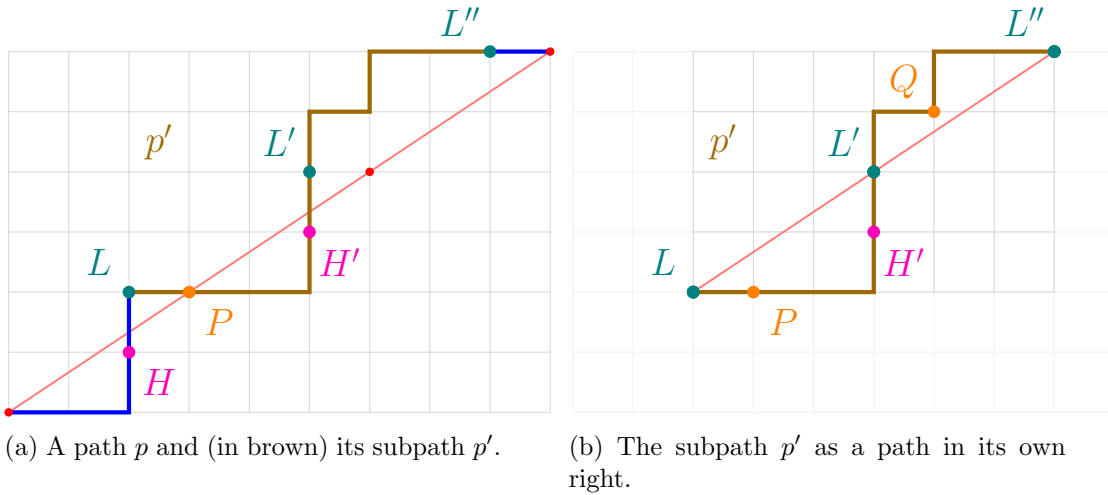

\begin{figure}[ht]
\centering
\begin{subfigure}{.4\textwidth}
  \centering
    \begin{tikzpicture}[scale = 0.6]
    \tikzmath{
    int \a, \b, \g;
    \a = 3;
    \b = 2;
    \g = 3;
    \p = 1;
    \q = 0;
    \d = \b*(\a*\q-\b*\p)/(\b^2+\a^2);
    \r = 1;
    }
    \draw[gray!30] (0,0) grid (\g * \a ,\g * \b);
    \draw[thick,red!60] (0,0) -- (\g * \a ,\g * \b);
    \draw[black,line width = 0.8pt, dotted]
        (\p-\r, \q+\r*\a/\b) -- (\p+\r+2*\d, \q-\r*\a/\b-2*\d*\a/\b);
    \draw[ultra thick, blue] (0,0)  -- (1,0);
    \draw[ultra thick, black]
    (1,0) -- (1,1) -- (1,2) -- (1,3) -- (2,3) -- (3,3) -- (4,3) -- (4,4) -- (5,4) -- (6,4) -- (6,5) -- (6,6) -- (7,6) -- (8,6) -- (9,6);
    \foreach \i in {0,...,\g}
        \fill[red] (\a * \i,\b * \i) circle (1.5pt);
    \fill[orange]	    
        (1,0) circle (5pt);
    \fill[teal]	    
        (1,1) circle (5pt)
        (4,3) circle (5pt);
    \fill[red]	    
        (0,0) circle (5pt)
        (6,4) circle (5pt)
        (9,6) circle (5pt);
    \draw[shift = {(1.5, -0.5)}, color = orange] node[font=\Large] {$H$};
    \draw[shift = {(3.5, 3.5)}, color = teal] node[font=\Large] {$L$};
    \draw[shift = {(0.5, -0.5)}, color = blue] node[font=\Large] {$p_1$};
    \draw[shift = {(5.5, 5.5)}, color = black] node[font=\Large] {$p_2$};
    \phantom{
    \fill[black]	    
        (-1,-1) circle (1pt)
        (\g * \a +1, \g * \b +1) circle (1pt);
    } 
    \end{tikzpicture}
    \caption{Original path with 12 flaws.}
  \label{fig:cyc_perm_observation1}
\end{subfigure}%
\qquad 
\pgfmathsetmacro{\xShiftpOne}{8}%
\pgfmathsetmacro{\yShiftpOne}{6}%
\pgfmathsetmacro{\xShiftpTwo}{-1}%
\pgfmathsetmacro{\yShiftpTwo}{0}%
\begin{subfigure}{.4\textwidth}
\centering
    \begin{tikzpicture}[scale = 0.6]
    \tikzmath{
    int \a, \b, \g;
    \a = 3;
    \b = 2;
    \g = 3;
    \p = 8;
    \q = 6;
    \d = \b*(\a*\q-\b*\p)/(\b^2+\a^2);
    \r = 0.4;
    }
    \draw[gray!30] (0,0) grid (\g * \a ,\g * \b);
    \draw[thick,red!60] (0,0) -- (\g * \a ,\g * \b);
    \draw[black,line width = 0.8pt, dotted]
        (\p-\r, \q+\r*\a/\b) -- (\p+\r+2*\d, \q-\r*\a/\b-2*\d*\a/\b);
    \draw[ultra thick, blue] 
    (\xShiftpOne+0,0+\yShiftpOne)  -- (\xShiftpOne+1,0+\yShiftpOne);
    \draw[ultra thick, black]
    (\xShiftpTwo+1,0+\yShiftpTwo) -- (\xShiftpTwo+1,1+\yShiftpTwo) -- (\xShiftpTwo+1,2+\yShiftpTwo) -- (\xShiftpTwo+1,3+\yShiftpTwo) -- (\xShiftpTwo+2,3+\yShiftpTwo) -- (\xShiftpTwo+3,3+\yShiftpTwo) -- 
    (\xShiftpTwo+4,3+\yShiftpTwo)-- (\xShiftpTwo+4,4+\yShiftpTwo) -- (\xShiftpTwo+5,4+\yShiftpTwo) -- (\xShiftpTwo+6,4+\yShiftpTwo) -- (\xShiftpTwo+6,5+\yShiftpTwo) -- (\xShiftpTwo+6,6+\yShiftpTwo) -- (\xShiftpTwo+7,6+\yShiftpTwo) -- (\xShiftpTwo+8,6+\yShiftpTwo) -- (\xShiftpTwo+9,6+\yShiftpTwo);
    \foreach \i in {0,...,\g}
        \fill[red] (\a * \i,\b * \i) circle (1.5pt);
    \fill[orange]	    
        (1+\xShiftpOne,\yShiftpOne+0) circle (5pt)
        (1+\xShiftpTwo,\yShiftpTwo+0) circle (5pt);
    \fill[teal]	    
        (\xShiftpTwo+1,1+\yShiftpTwo) circle (5pt)
        (\xShiftpTwo+4,3+\yShiftpTwo) circle (5pt);
    \fill[red]	    
        (0+\xShiftpOne,\yShiftpOne+0) circle (5pt)
        (6+\xShiftpTwo,\yShiftpTwo+4) circle (5pt)
        (9+\xShiftpTwo,\yShiftpTwo+6) circle (5pt);
    \draw[shift = {(0.5+\xShiftpOne, \yShiftpOne+ 0.5)}, color = blue] node[font=\Large] {$p_1$};
    \draw[shift = {(5.5+\xShiftpTwo, \yShiftpTwo+ 5.5)}, color = black] node[font=\Large] {$p_2$};
    \phantom{
    \fill[black]	    
        (-1,-1) circle (1pt)
        (\g * \a +1, \g * \b +1) circle (1pt);
    } 
    \end{tikzpicture}
    \caption{Rearranged path with 14 flaws.}
\label{fig:cyc_perm_observation2}
\end{subfigure}%
\\[2em]
\begin{subfigure}{.4\textwidth}
  \centering
    \begin{tikzpicture}[scale = 0.6]
    \tikzmath{
    int \a, \b, \g;
    \a = 3;
    \b = 2;
    \g = 3;
    \p = 4;
    \q = 3;
    \d = \b*(\a*\q-\b*\p)/(\b^2+\a^2);
    \r = 0.5;
    }
    \draw[gray!30] (0,0) grid (\g * \a ,\g * \b);
    \draw[thick,red!60] (0,0) -- (\g * \a ,\g * \b);
    \draw[black,line width = 0.8pt, dotted]
        (\p-\r, \q+\r*\a/\b) -- (\p+\r+2*\d, \q-\r*\a/\b-2*\d*\a/\b);
    \draw[ultra thick, blue] (0,0)  -- (1,0) -- (1,1) -- (1,2) -- (1,3) -- (2,3) -- (3,3) -- (4,3);
    \draw[ultra thick, black]
    (4,3)-- (4,4) -- (5,4) -- (6,4) -- (6,5) -- (6,6) -- (7,6) -- (8,6) -- (9,6);
    \foreach \i in {0,...,\g}
        \fill[red] (\a * \i,\b * \i) circle (1.5pt);
    \fill[orange]	    
        (1,0) circle (5pt);
    \fill[teal]	    
        (1,1) circle (5pt)
        (4,3) circle (5pt);
    \fill[red]	    
        (0,0) circle (5pt)
        (6,4) circle (5pt)
        (9,6) circle (5pt);
    \draw[shift = {(1.5, -0.5)}, color = orange] node[font=\Large] {$H$};
    \draw[shift = {(3.5, 3.5)}, color = teal] node[font=\Large] {$L$};
    \draw[shift = {(0.5, 2.5)}, color = blue] node[font=\Large] {$p_1$};
    \draw[shift = {(5.5, 5.5)}, color = black] node[font=\Large] {$p_2$};
    \phantom{
    \fill[black]	    
        (-1,-1) circle (1pt)
        (\g * \a +1, \g * \b +1) circle (1pt);
    } 
    \end{tikzpicture}
    \caption{Original path with 12 flaws.}
  \label{fig:cyc_perm_observation3}
\end{subfigure}%
\qquad 
\pgfmathsetmacro{\xShiftpOne}{5}%
\pgfmathsetmacro{\yShiftpOne}{3}%
\pgfmathsetmacro{\xShiftpTwo}{-4}%
\pgfmathsetmacro{\yShiftpTwo}{-3}%
\begin{subfigure}{.4\textwidth}
\centering
    \begin{tikzpicture}[scale = 0.6]
    \tikzmath{
    int \a, \b, \g;
    \a = 3;
    \b = 2;
    \g = 3;
    \p = 5;
    \q = 3;
    \d = \b*(\a*\q-\b*\p)/(\b^2+\a^2);
    \r = 0.8;
    }
    \draw[gray!30] (0,0) grid (\g * \a ,\g * \b);
    \draw[thick,red!60] (0,0) -- (\g * \a ,\g * \b);
    \draw[black,line width = 0.8pt, dotted]
        (\p-\r, \q+\r*\a/\b) -- (\p+\r+2*\d, \q-\r*\a/\b-2*\d*\a/\b);
    \draw[ultra thick, blue] 
    (\xShiftpOne+0,0+\yShiftpOne)  -- (\xShiftpOne+1,0+\yShiftpOne) -- (\xShiftpOne+1,1+\yShiftpOne) -- (\xShiftpOne+1,2+\yShiftpOne) -- (\xShiftpOne+1,3+\yShiftpOne) -- (\xShiftpOne+2,3+\yShiftpOne) -- (\xShiftpOne+3,3+\yShiftpOne) -- (\xShiftpOne+4,3+\yShiftpOne);
    \draw[ultra thick, black]
    (\xShiftpTwo+4,3+\yShiftpTwo)-- (\xShiftpTwo+4,4+\yShiftpTwo) -- (\xShiftpTwo+5,4+\yShiftpTwo) -- (\xShiftpTwo+6,4+\yShiftpTwo) -- (\xShiftpTwo+6,5+\yShiftpTwo) -- (\xShiftpTwo+6,6+\yShiftpTwo) -- (\xShiftpTwo+7,6+\yShiftpTwo) -- (\xShiftpTwo+8,6+\yShiftpTwo) -- (\xShiftpTwo+9,6+\yShiftpTwo);
    \foreach \i in {0,...,\g}
        \fill[red] (\a * \i,\b * \i) circle (1.5pt);
    \fill[orange]	    
        (1+\xShiftpOne,\yShiftpOne+0) circle (5pt);
    \fill[teal]	    
        (\xShiftpOne+1,1+\yShiftpOne) circle (5pt)
        (\xShiftpOne+4,3+\yShiftpOne) circle (5pt)
        (\xShiftpTwo+4,3+\yShiftpTwo) circle (5pt);
    \fill[red]	    
        (0+\xShiftpOne,\yShiftpOne+0) circle (5pt)
        (6+\xShiftpTwo,\yShiftpTwo+4) circle (5pt)
        (9+\xShiftpTwo,\yShiftpTwo+6) circle (5pt);
    \draw[shift = {(0.5+\xShiftpOne, \yShiftpOne+2.5)}, color = blue] node[font=\Large] {$p_1$};
    \draw[shift = {(5.5+\xShiftpTwo, \yShiftpTwo+ 5.5)}, color = black] node[font=\Large] {$p_2$};
    \phantom{
    \fill[black]	    
        (-1,-1) circle (1pt)
        (\g * \a +1, \g * \b +1) circle (1pt);
    } 
    \end{tikzpicture}
    \caption{Rearranged path with 10 flaws.}
\label{fig:cyc_perm_observation4}
\end{subfigure}%
\caption{
Rearrangement of the subpaths of a path 
with $\beta =1$ interior boundary points and $\lambda = 2$ LPAs
changes the number of flaws, according to \cref{obs:flaws}. 
Splitting the path at an HPB maps diagram (a) to diagram~(b);
splitting the same path at an LPA maps diagram (c) to diagram~(d).
}
  \label{fig:cyc_perm_observation}
\end{figure}

We now define the subsets $X$ and $Y$ of ${W_k(g) \setminus S_k(g)}$ by reference to an arbitrary path $p \in {W_k(g) \setminus S_k(g)}$.
Split $p$ at its last non-terminal boundary point into $q r$, and regard $q$ and $r$ as paths in their own right.  If $p$ has no interior boundary points, then $p$ is split at its startpoint and $q$ is empty. 
Since $p \notin S_k(g)$, either $q$ has at least one flaw or $r$ has non-max flaws; in the latter case, $r$ has at least one HPB because $p$ splits at its last non-terminal boundary point into $q r$ and so $r$ itself has no interior boundary points.  Therefore exactly one of three cases holds:
\begin{description}
    \item[Case 1:] 
    $q$ has no flaws and $r$ has non-max flaws.
    Then $p \in X$.
        
    \item[Case 2:] 
    $q$ has at least one flaw and $r$ has max flaws.   
    Then $p \in Y$.        
    \item[Case 3:] 
    $q$ has at least one flaw and $r$ has non-max flaws.
    If the LPAs of $q$ are closer to the boundary of $p$ than are the HPBs of $r$, then $p \in Y$. Otherwise $p \in X$.
\end{description}

See \cref{fig:domain} for an illustration of Case~$3$.

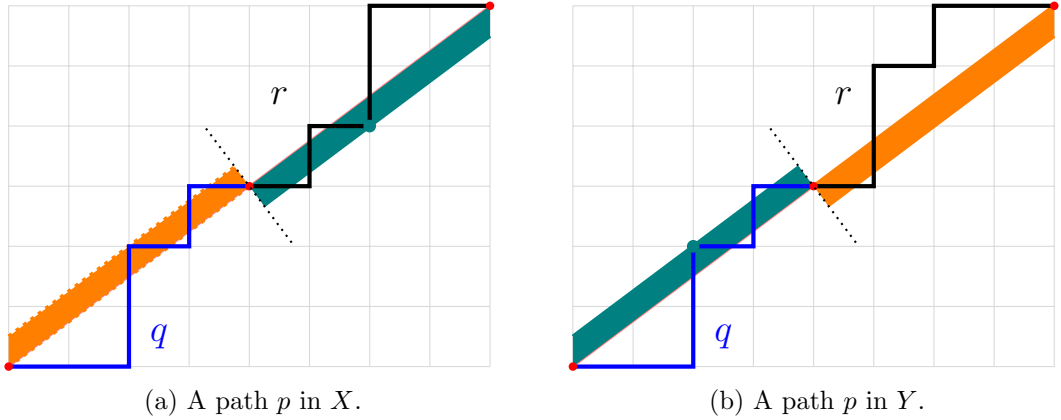
\begin{figure}[ht]
\begin{center}
\begin{subfigure}{.4\textwidth}
        \begin{tikzpicture}[scale = 0.8]
        \draw[gray!30] (0,0) grid (8,6);
        \draw[thick,red!50] (0,0) -- (8,6);
        \fill [opacity = 0.2, orange]
            (0,0) -- (0,1/2) -- (4-6/25,3+8/25) -- (4,3) -- cycle;
        \draw[orange, line width = 1.5pt, dotted]
            (0,0) -- (4,3);
        \draw[orange, line width = 1.5pt, dotted]
            (0,1/2) -- (4-6/25,3+8/25);
        \draw[teal, line width = 1.5pt]
            (6,4) -- (8, 6-1/2);
        \draw[teal, line width = 1.5pt, dotted]
            (4+6/25,3-8/25) -- (6,4);
        \fill [opacity = 0.2, teal]
            (4,3) -- (4+6/25,3-8/25) -- (8, 6-1/2) -- (8,6) -- cycle;
        \draw[ultra thick, blue] 
            (0,0)  -- (1,0) -- (2,0) -- (2,1) -- (2,2) -- (3,2) -- (3,3) -- (4,3);
        \draw[ultra thick, black] 
            (4,3) -- (5,3) -- (5,4) -- (6,4) -- (6,5) -- (6,6) -- (7,6) -- (8,6);
        \fill[red]  
            (0,0) circle (2pt)
            (4,3) circle (2pt)
            (8,6) circle (2pt);
        \draw[black,line width = 0.8pt, dotted]
            (4-3*6/25,3+3*8/25) -- (4+3*6/25,3-3*8/25);
        \fill[teal]	
            (6,4) circle (3pt);
        \draw[shift = {(2.5,0.5)}, color = blue] node[font=\Large] {$q$};
        \draw[shift = {(4.5,4.5)}, color = black] node[font=\Large] {$r$};
    \end{tikzpicture}
    \caption{A path $p$ in $X$.}
    \label{fig:domain_X}
\end{subfigure}
\qquad
\begin{subfigure}{.4\textwidth}
        \begin{tikzpicture}[scale = 0.8]
        \draw[gray!30] (0,0) grid (8,6);
        \draw[thick,red!50] (0,0) -- (8,6);
        \fill [opacity = 0.2, orange]
            (4,3) -- (4+6/25,3-8/25) -- (8, 6-1/2) -- (8,6) -- cycle;
        \draw[orange, line width = 1.5pt]
            (4+6/25,3-8/25) -- (8, 6-1/2);
        \draw[orange, line width = 1.5pt]
            (4,3) -- (8, 6);
        \draw[teal, line width = 1.5pt]
            (2,2) -- (4-6/25,3+8/25);
        \draw[teal, line width = 1.5pt, dotted]
            (0,1/2) -- (2,2);
        \fill [opacity = 0.2, teal]
        (0,0) -- (0,1/2) -- (4-6/25,3+8/25) -- (4,3) -- cycle;
        \draw[ultra thick, blue] 
            (0,0)  -- (1,0) -- (2,0) -- (2,1) -- (2,2) -- (3,2) -- (3,3) -- (4,3);
        \draw[ultra thick, black] 
            (4,3) -- (5,3) -- (5,4) -- (5,5) -- (6,5) -- (6,6) -- (7,6) -- (8,6);
        \fill[red]  
            (0,0) circle (2pt)
            (4,3) circle (2pt)
            (8,6) circle (2pt);
        \draw[black,line width = 0.8pt, dotted]
            (4-3*6/25,3+3*8/25) -- (4+3*6/25,3-3*8/25);
        \fill[teal]	
            (2,2) circle (3pt);
        \draw[shift = {(2.5,0.5)}, color = blue] node[font=\Large] {$q$};
        \draw[shift = {(4.5,4.5)}, color = black] node[font=\Large] {$r$};
    \end{tikzpicture}
    \caption{A path $p$ in $Y$.}
    \label{fig:domain_Y}
\end{subfigure}
\caption{Let $(a,b) = (4,3)$ and split the path $p$ into~$qr$ at its last non-terminal boundary point.
In diagram~(a), we have $p \in W_6(2)$ and the green region is determined by the elevation of the HPBs of the subpath~$r$; this in turn determines an open orange ``forbidden region'' that the subpath $q$ must avoid so that $p \in X$.
In diagram~(b), we have $p \in W_7(2)$ and the green region is determined by the elevation of the LPAs of the subpath~$q$; this in turn determines a closed orange ``forbidden region'' that the subpath $r$ must avoid so that $p \in Y$.}
\label{fig:domain}
\end{center}
\end{figure}

We now give a more concise definition of the subsets $X$ and~$Y$. Recall that $k$ is fixed and satisfies $0 \leq k < g(a+b)-1$ throughout this section.

\begin{definition}[The subsets $X$ and $Y$]
\phantomsection
\label{def:domain_partition}
    Let $p \in W_k(g) \setminus S_k(g)$. Split $p$ at its last non-terminal boundary point into~$q r$, and regard $q$ and $r$ as paths.  
    The path $p$ lies in $Y$ provided:
    \begin{enumerate}[label=\textit{(\roman*)}]
        \item $q$ has at least one flaw, and
        \label{def:domain_partition:condition_1}
        
        \item the elevation of the LPAs of $q$ is smaller than the magnitude of the elevation of the HPBs of~$r$ (if any). 
        \label{def:domain_partition:condition_2}
    \end{enumerate}
    Otherwise, $p$ lies in~$X$. \lrcornerqed
\end{definition}
Note that $Y$ is empty if $k=0$.  We now use \cref{def:domain_partition} to specify a canonical representation for a path in each of $X$ and $Y$ as a concatenation of paths.

\begin{definition}[Canonical representation of paths in $X$ and~$Y$]
\label{def:domain_split}
    Let $p \in W_k(g) \setminus S_k(g)$. Split $p$ at its last non-terminal boundary point into $p = q r$.
    \begin{description}

            \item[Case $p \in X$:]
            the path $r$ has at least one HPB. Split $r$ at its last HPB into~$r = r_1 r_2$. The canonical representation of $p$ is~$q r_1 r_2$.
    
        \item[Case $p \in Y$:]
            the path $q$ has at least one LPA. Split $q$ at its last LPA into~$q = q_1 q_2$. The canonical representation of $p$ is~$q_1 q_2 r$. \lrcornerqed
    \end{description}
\end{definition}

We now define the subsets $\mathcal{X}$ and~$\mathcal{Y}$ of $W_{k+1}(g)$ by reference to an arbitrary path $\mathbbm{p} \in W_{k+1}(g)$. 
Since $\mathbbm{p}$ has at least one flaw, it has at least one LPA.
Throughout, we use regular typeface (for example~$p$) for a path originating in $W_k(g)$ whereas we use blackboard bold typeface (for example~$\mathbbm{p}$) for a path originating in~$W_{k+1}(g)$.
This is intended to help distinguish the domain and codomain of the function~$\phi$ (namely $W_k(g)$ and~$W_{k+1}(g)$, respectively) from the domain and codomain of the function~$\psi$ (namely $W_{k+1}(g)$ and~$W_k(g)$, respectively).

\begin{definition}[The subsets $\mathcal{X}$ and $\mathcal{Y}$]
\phantomsection
\label{def:codomain_partition}
    Let $\mathbbm{p} \in W_{k+1}(g)$. 
    The path $\mathbbm{p}$ lies in $\mathcal{Y}$ provided: 
    \begin{enumerate}[label=\textit{(\roman*)}]
        \item $\mathbbm{p}$ has at least two LPAs, and
        \label{def:codomain_partition:condition_1}
        
        \item the subpath of $\mathbbm p$ lying between the last two LPAs of $\mathbbm{p}$ contains no boundary points of~$\mathbbm{p}$.
        \label{def:codomain_partition:condition_2}
    \end{enumerate}
    Otherwise, $\mathbbm{p}$ lies in~$\mathcal{X}$. \lrcornerqed
\end{definition}
Note that $\mathcal{Y}$ is empty if $k=0$.
See \cref{fig:codomain} for an illustration of \cref{def:codomain_partition}.

We now use \cref{def:codomain_partition} to specify a canonical path split for a path in each of $\mathcal X$ and~$\mathcal Y$.

\begin{definition}[Canonical representation of paths in $\mathcal X$ and~$\mathcal Y$]
\phantomsection
\label{def:codomain_split}
    Let $\mathbbm{p} \in W_{k+1}(g)$.
    \begin{description}

        \item[Case $\mathbbm{p} \in \mathcal{X}$:]
            let $L$ be the last LPA of~$\mathbbm{p}$, and let $B$ be the boundary point of $\mathbbm{p}$ (possibly the startpoint of $\mathbbm{p}$) which immediately precedes~$L$. 
            Split $\mathbbm{p}$ at $B$ and $L$ into $\mathbbm{p} = \mathbbm{q} \mathbbm{r}_2 \mathbbm{r}_1$. 
            The canonical representation of $\mathbbm{p}$ is~$\mathbbm{q} \mathbbm{r}_2 \mathbbm{r}_1$.

        \item[Case $\mathbbm{p} \in \mathcal{Y}$:]
            the path $\mathbbm{p}$ has at least two LPAs. 
            Split $\mathbbm{p}$ at its 
            last two LPAs
            into $\mathbbm{p} = \mathbbm{q}_1 \mathbbm{r} \mathbbm{q}_2$. 
            The canonical representation of $\mathbbm{p}$ is~$\mathbbm{q}_1 \mathbbm{r} \mathbbm{q}_2$.\lrcornerqed
    \end{description}
\end{definition}

\begin{figure}[ht]
\begin{center}
\begin{subfigure}{.4\textwidth}
        \begin{tikzpicture}[scale = 0.8]
        \draw[gray!30] (0,0) grid (8,6);
        \draw[thick,red!50] (0,0) -- (8,6);
        \fill [opacity = 0.2, orange]
            (2,2) -- (2+6/25,2-8/25) -- (6+6/25, 5-8/25) -- (6,5) -- cycle;
        \draw[orange, line width = 1.5pt]
            (2,2) -- (6,5);
        \draw[orange, line width = 1.5pt, dotted]
            (2+6/25,2-8/25) -- (6+6/25, 5-8/25);
        \fill [opacity = 0.2, teal]
            (0,0) -- (0,1/2) -- (2,2) -- (2+6/25,2-8/25) -- cycle;
        \draw[teal, line width = 1.5pt, dotted]
            (0,1/2) -- (2,2);
        \draw[teal, line width = 1.5pt, dotted]
            (0,0) -- (2+6/25,2-8/25);
        \fill [opacity = 0.2, teal]
            (6,5) -- (8-2/3,6) -- (8,6) -- (6+6/25, 5-8/25) -- cycle;
        \draw[teal, line width = 1.5pt]
            (6,5) -- (8-2/3,6);
        \draw[teal, line width = 1.5pt, dotted]
          (6+6/25, 5-8/25) -- (8,6);
        \draw[ultra thick, blue] 
            (0,0)  -- (1,0) -- (2,0) -- (2,1) -- (2,2);
        \draw[ultra thick, black] 
            (2,2) -- (3,2) -- (3,3) -- (4,3);
        \draw[ultra thick, black] 
            (4,3) -- (4,4) -- (4,5) -- (5,5) -- (6,5);
        \draw[ultra thick, purple] 
            (6,5) -- (7,5) -- (7,6) -- (8,6);
        \fill[red]  
            (0,0) circle (2pt)
            (4,3) circle (2pt)
            (8,6) circle (2pt);
        \draw[black,line width = 0.8pt, dotted]
            (2-2*6/25,2+2*8/25) -- (2+4*6/25,2-4*8/25);
        \draw[black,line width = 0.8pt, dotted]
            (6-2*6/25,5+2*8/25) -- (6+4*6/25,5-4*8/25);
        \fill[teal]	
            (2,2) circle (3pt)
            (6,5) circle (3pt);
        \fill[red]  
            (4,3) circle (3pt);
        \draw[shift = {(3.5,4.5)}, color = black] node[font=\Large] {$\mathbbm{r}$};
    \end{tikzpicture}
    \caption{A path $\mathbbm p$ in $\mathcal{X}$.}
    \label{fig:codomain_X}
\end{subfigure}
\qquad
%
\begin{subfigure}{.4\textwidth}
        \begin{tikzpicture}[scale = 0.8]
        \draw[gray!30] (0,0) grid (8,6);
        \draw[thick,red!50] (0,0) -- (8,6);
        \fill [opacity = 0.2, orange]
            (2,2) -- (2+6/25,2-8/25) -- (6+6/25, 5-8/25) -- (6,5) -- cycle;
        \draw[orange, line width = 1.5pt]
            (2+6/25,2-8/25) -- (6+6/25, 5-8/25);
        \draw[orange, line width = 1.5pt]
            (2,2) -- (6, 5);
        \draw[teal, line width = 1.5pt, dotted]
            (0,0) -- (2+6/25,2-8/25);
        \draw[teal, line width = 1.5pt, dotted]
            (0,1/2) -- (2,2);
        \fill [opacity = 0.2, teal]
            (0,0) -- (0,1/2) -- (2,2) -- (2+6/25,2-8/25) -- cycle;
        \draw[teal, line width = 1.5pt]
            (6,5) -- (8-2/3,6);
        \fill [opacity = 0.2, teal]
            (6,5) -- (8-2/3,6) -- (8,6) -- (6+6/25, 5-8/25) -- cycle;
        \draw[teal, line width = 1.5pt, dotted]
          (6+6/25, 5-8/25) -- (8,6);
        \draw[ultra thick, blue] 
            (0,0)  -- (1,0) -- (2,0) -- (2,1) -- (2,2);
        \draw[ultra thick, black]
            (2,2) -- (3,2) -- (3,3) -- (3,4) -- (4,4) -- (4,5) -- (5,5) -- (6,5);
        \draw[ultra thick, purple] 
            (6,5) -- (7,5) -- (7,6) -- (8,6);
        \fill[red]  
            (0,0) circle (2pt)
            (4,3) circle (2pt)
            (8,6) circle (2pt);
        \draw[black,line width = 0.8pt, dotted]
            (2-2*6/25,2+2*8/25) -- (2+4*6/25,2-4*8/25);
        \draw[black,line width = 0.8pt, dotted]
            (6-2*6/25,5+2*8/25) -- (6+4*6/25,5-4*8/25);
        \fill[teal]	
            (2,2) circle (3pt)
            (6,5) circle (3pt);
        \draw[shift = {(2.5,4.5)}, color = black] node[font=\Large] {$\mathbbm{r}$};
    \end{tikzpicture}
    \caption{A path $\mathbbm p$ in $\mathcal{Y}$.}
    \label{fig:codomain_Y}
\end{subfigure}
\caption{Let $(a,b) = (4,3)$.
In diagram (a), we have $\mathbbm p \in W_7(2)$ and the subpath $\mathbbm r$ between the (last) two LPAs of $\mathbbm p$ contains a boundary point of~$\mathbbm p$.
In diagram~(b), we have $\mathbbm p \in W_8(2)$ and the subpath $\mathbbm r$ between the (last) two LPAs of $\mathbbm p$ contains no boundary points of~$\mathbbm p$.
}
\label{fig:codomain}
\end{center}
\end{figure}
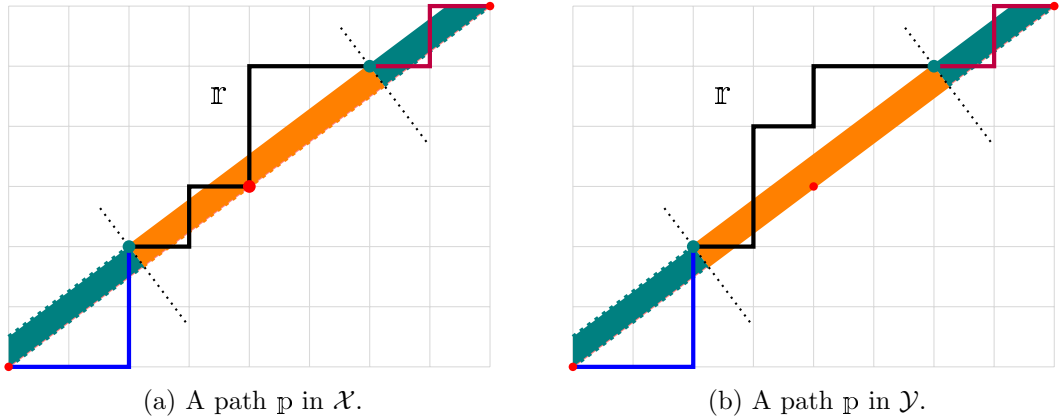

\subsection{The actions of $\phi^X$, $\phi^Y$, $\psi^{\mathcal{X}}$ and $\psi^\mathcal{Y}$}

We now define the maps $\phi^X, \phi^Y$, $\psi^\mathcal{X}$ and $\psi^\mathcal{Y}$, whose domains and codomains are given in \cref{fig:phi_construction_new}.  Illustrations of these maps are given in \cref{fig:XY_action}.

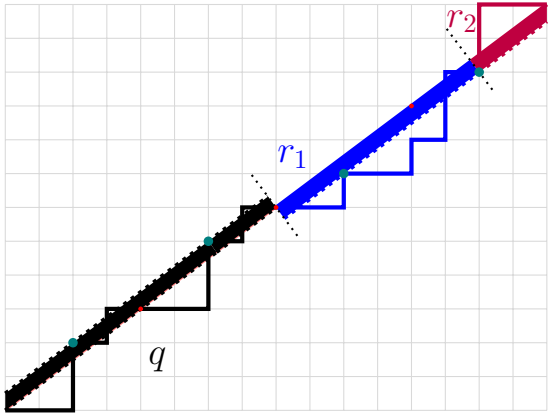
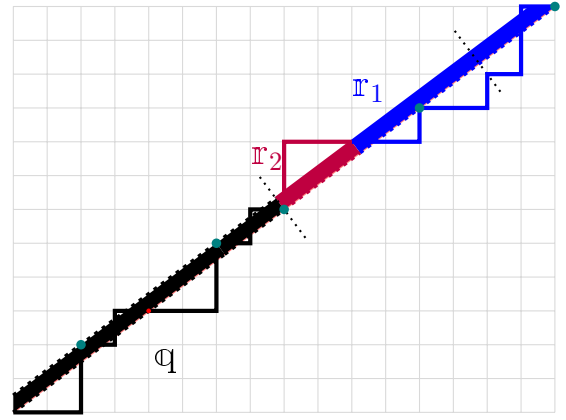
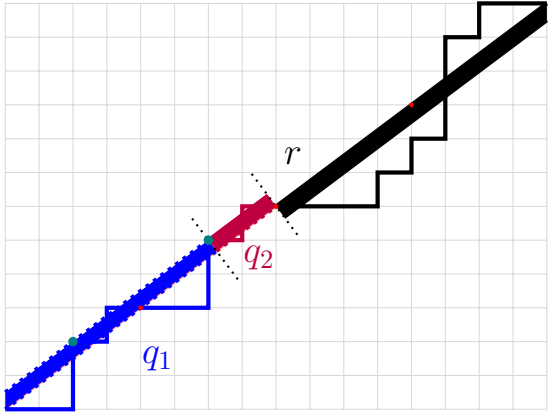
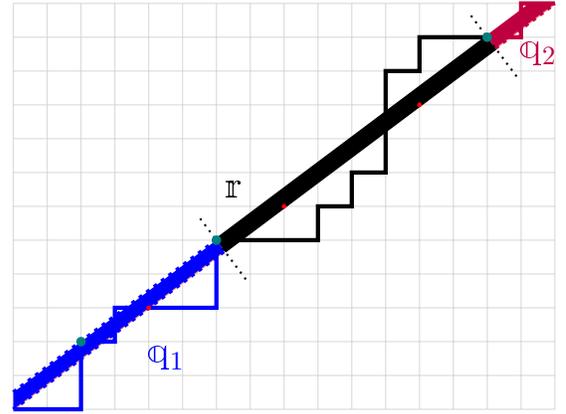
\begin{figure}[h!t]
\begin{center}
\begin{subfigure}{.45\textwidth}
    \begin{tikzpicture}[scale = 0.45]
        \draw[gray!30] (0,0) grid (16,12);
        \draw[thick,red!50] (0,0) -- (16,12);
        \fill[opacity = 0.2, blue]
            (8,6) -- (8+6/25,6-8/25) -- (14,10) --  (14-2*3/25,10+2*4/25) -- cycle;
        \draw[blue, line width = 1.5pt, opacity = 0.4]
            (8,6) -- (14-2*3/25,10+2*4/25);    
        \draw[blue, line width = 1.5pt, opacity = 0.4, dotted]
            (8+6/25,6-8/25) -- (14,10);
        %
        \fill[opacity = 0.2, purple]
            (14-2*3/25,10+2*4/25) -- (14,10) -- (16, 12-1/2) -- (16,12) -- cycle;
        \draw[purple, line width = 1.5pt, opacity = 0.4]
            (14-2*3/25,10+2*4/25) -- (16,12);
        \draw[purple, line width = 1.5pt, opacity = 0.4, dotted]
            (14,10) -- (16, 12-1/2);
        %
        \fill[opacity = 0.2, black]
            (0,0) -- (0,1/2) -- (6,5) -- (6+6/25,5-8/25) -- cycle;
        \draw[black, line width = 1.5pt, dotted, opacity = 0.4]
            (0,1/2) -- (6,5); 
        \draw[black, line width = 1.5pt, dotted, opacity = 0.4]
            (0,0) -- (6+ 6/25,5-8/25);
        \fill[opacity = 0.2, black]
            (6,5) -- (8-6/25,6+8/25) -- (8,6) -- (6+6/25,5-8/25) -- cycle;
        \draw[black, line width = 1.5pt, opacity = 0.4, dotted]
            (6,5) -- (8-6/25,6+8/25);    
        \draw[black, line width = 1.5pt, dotted, opacity = 0.4]
            (6+ 6/25,5-8/25) -- (8,6);
        \draw[black,line width = 0.8pt, dotted]
            (8-3*6/25,6+3*8/25) -- (8+3*6/25,6-3*8/25);
        \draw[black,line width = 0.8pt, dotted]
            (14-8*3/25,10+8*4/25) -- (14+4*3/25,10-4*4/25); 
        \draw[ultra thick, black] 
            (0,0)  -- (1,0) -- (2,0) -- (2,1) -- (2,2) -- (3,2) -- (3,3) -- (4,3);
        \draw[ultra thick, black]
            (4,3)  -- (5,3) -- (6,3) -- (6,4) -- (6,5);
        \draw[ultra thick, black] 
            (6,5) -- (7,5) -- (7,6) -- (8,6);
        \draw[ultra thick, blue] 
            (8,6) -- (9,6) -- (10,6) -- (10,7) -- (11,7) -- (12,7) -- (12,8) -- (13,8) -- (13,9)-- (13,10) -- (14,10);
        \draw[ultra thick, purple] 
            (14,10) -- (14,11) -- (14,12) -- (15,12) -- (16,12);
        \fill[red]  
            (4,3) circle (2pt)
            (8,6) circle (2pt)
            (12,9) circle (2pt);
        \fill[teal]	
            (2,2) circle (4pt);
        \fill[teal]	
            (6,5) circle (4pt);
        \fill[teal]	
            (10,7) circle (4pt);
        \fill[teal]	
            (14,10) circle (4pt);
        \draw[shift = {(4.5,1.5)}, color = black]        
            node[font=\Large] {$q$};
        \draw[shift = {(8.5,7.5)}, color = blue] 
            node[font=\Large] {$r_1$};
        \draw[shift = {(13.5,11.5)}, color = purple] 
            node[font=\Large] {$r_2$};
        \draw[shift = {(14.5,9.5)}, color = teal] 
            node[font=\Large] {$H$};
    \end{tikzpicture}
    \caption{Canonical representation of a path $p$ in~$X$.}
    \label{fig:X_preimage}
\end{subfigure}
\quad \quad
%
%
\pgfmathsetmacro{\xShiftQ}{0}
\pgfmathsetmacro{\yShiftQ}{0}
\pgfmathsetmacro{\xShiftrOne}{2}
\pgfmathsetmacro{\yShiftrOne}{2}
\pgfmathsetmacro{\xShiftrTwo}{-6}
\pgfmathsetmacro{\yShiftrTwo}{-4}
\begin{subfigure}{.45\textwidth}
    \begin{tikzpicture}[scale = 0.45]
        \draw[gray!30] (0,0) grid (16,12);
        \draw[thick,red!50] (0,0) -- (16,12);
        \fill[opacity = 0.2, blue]
            (\xShiftrOne+8,6+\yShiftrOne) -- (\xShiftrOne+8+6/25,6-8/25+\yShiftrOne) -- (\xShiftrOne+14,10+\yShiftrOne) --  (\xShiftrOne+14-2/3,10+\yShiftrOne) -- cycle;
        \draw[blue, line width = 1.5pt, opacity = 0.4]
            (\xShiftrOne+8,6+\yShiftrOne) -- (\xShiftrOne+14-2/3,10+\yShiftrOne);    
        \draw[blue, line width = 1.5pt, opacity = 0.4, dotted]
            (\xShiftrOne+8+6/25,6-8/25+\yShiftrOne) -- (\xShiftrOne+14,10+\yShiftrOne);
        %
        \fill[opacity = 0.2, purple]
            (\xShiftrTwo+14-2*3/25,10+2*4/25+\yShiftrTwo) -- (\xShiftrTwo+14,10+\yShiftrTwo) -- (\xShiftrTwo+16+2*3/25, 12-2*4/25+\yShiftrTwo) -- (\xShiftrTwo+16,12+\yShiftrTwo) -- cycle;
        \draw[purple, line width = 1.5pt, opacity = 0.4]
            (\xShiftrTwo+14-2*3/25,10+2*4/25+\yShiftrTwo) -- (\xShiftrTwo+16,12+\yShiftrTwo);
        \draw[purple, line width = 1.5pt, opacity = 0.4, dotted]
            (\xShiftrTwo+14,10+\yShiftrTwo) -- (\xShiftrTwo+16+2*3/25, 12-2*4/25+\yShiftrTwo);
        %
        \fill[opacity = 0.2, black]
            (0,0) -- (0,1/2) -- (6,5) -- (6+6/25,5-8/25) -- cycle;
        \draw[black, line width = 1.5pt, dotted, opacity = 0.4]
            (0,1/2) -- (6,5);   
        \draw[black, line width = 1.5pt, dotted, opacity = 0.4]
            (0,0) -- (6+ 6/25,5-8/25);
        \fill[opacity = 0.2, black]
            (6,5) -- (8-6/25,6+8/25) -- (8,6) -- (6+6/25,5-8/25) -- cycle;
        \draw[black, line width = 1.5pt, opacity = 0.4, dotted]
            (6,5) -- (8-6/25,6+8/25);    
        \draw[black, line width = 1.5pt, dotted, opacity = 0.4]
            (6+ 6/25,5-8/25) -- (8,6);
        \draw[black,line width = 0.8pt, dotted]
            (8-3*6/25,6+3*8/25) -- (8+3*6/25,6-3*8/25);
        \draw[black,line width = 0.8pt, dotted]
            (14-8*3/25,10+8*4/25) -- (14+4*3/25,10-4*4/25); 
        \draw[ultra thick, black] 
            (0,0)  -- (1,0) -- (2,0) -- (2,1) -- (2,2) -- (3,2) -- (3,3) -- (4,3);
        \draw[ultra thick, black]
            (4,3)  -- (5,3) -- (6,3) -- (6,4) -- (6,5);
        \draw[ultra thick, black] 
            (6,5) -- (7,5) -- (7,6) -- (8,6);
        \draw[ultra thick, blue] 
            (\xShiftrOne+8,6+\yShiftrOne) -- (\xShiftrOne+9,6+\yShiftrOne) -- (\xShiftrOne+10,6+\yShiftrOne) -- (\xShiftrOne+10,7+\yShiftrOne) -- (\xShiftrOne+11,7+\yShiftrOne) -- (\xShiftrOne+12,7+\yShiftrOne) -- (\xShiftrOne+12,8+\yShiftrOne) -- (\xShiftrOne+13,8+\yShiftrOne) -- (\xShiftrOne+13,9+\yShiftrOne)-- (\xShiftrOne+13,10+\yShiftrOne) -- (\xShiftrOne+14,10+\yShiftrOne);
        \draw[ultra thick, purple] 
            (\xShiftrTwo+14,10+\yShiftrTwo) -- (\xShiftrTwo+14,11+\yShiftrTwo) -- (\xShiftrTwo+14,12+\yShiftrTwo) -- (\xShiftrTwo+15,12+\yShiftrTwo) -- (\xShiftrTwo+16,12+\yShiftrTwo);
        \fill[red]  
            (4,3) circle (2pt)
            (8,6) circle (2pt)
            (12,9) circle (2pt);
        \fill[teal]	
            (2,2) circle (4pt);
        \fill[teal]	
            (6,5) circle (4pt);
        \fill[teal]	
            (\xShiftrOne+10,7+\yShiftrOne) circle (4pt);
        \fill[teal]	
            (\xShiftrOne+14,10+\yShiftrOne) circle (4pt);
        \fill[teal]	
            (\xShiftrTwo+14,10+\yShiftrTwo) circle (4pt);
        \fill[teal]	
            (\xShiftrOne+8,6+\yShiftrOne) circle (4pt);
        \draw[shift = {(4.5,1.5)}, color = black]        
            node[font=\Large] {$\mathbbm{q}$};
        \draw[shift = {(\xShiftrOne+8.5,7.5+\yShiftrOne)}, color = blue] 
            node[font=\Large] {$\mathbbm{r}_1$};
        \draw[shift = {(\xShiftrTwo+13.5,11.5+\yShiftrTwo)}, color = purple] 
            node[font=\Large] {$\mathbbm{r}_2$};
        \draw[shift = {(\xShiftrOne+7.5,6.5+\yShiftrOne)}, color = teal] 
            node[font=\Large] {$L$};
    \end{tikzpicture}
    \caption{Canonical representation of a path $\mathbbm{p}$ in~$\mathcal{X}$.}
    \label{fig:X_image}
\end{subfigure}
\\[4em]
\begin{subfigure}{.45\textwidth}
    \begin{tikzpicture}[scale = 0.45]
        \draw[gray!30] (0,0) grid (16,12);
        \draw[thick,red!50] (0,0) -- (16,12);
        \fill[opacity = 0.2, black]
            (8,6) -- (8+6/25,6-8/25) -- (16, 12-1/2) -- (16,12) -- cycle;
        \draw[black, line width = 1.5pt, opacity = 0.4]
            (8+6/25,6-8/25) -- (16, 12-1/2);
        \draw[black, line width = 1.5pt, opacity = 0.4]
            (8,6) -- (16,12);
        \draw[blue, line width = 1.5pt, dotted, opacity = 0.5]
            (0,1/2) -- (6,5);   
        \fill[opacity = 0.2, blue]
            (0,0) -- (0,1/2) -- (6,5) -- (6+6/25,5-8/25) -- cycle;
        \draw[blue, line width = 1.5pt, dotted, opacity = 0.5]
            (0,0) -- (6+ 6/25,5-8/25);
        \draw[purple, line width = 1.5pt, opacity = 0.5]
            (6,5) -- (8-6/25,6+8/25);    
        \fill[opacity = 0.2, purple]
            (6,5) -- (8-6/25,6+8/25) -- (8,6) -- (6+6/25,5-8/25) -- cycle;
        \draw[purple, line width = 1.5pt, dotted, opacity = 0.5]
            (6+ 6/25,5-8/25) -- (8,6);
        \draw[black,line width = 0.8pt, dotted]
            (6-2*6/25,5+2*8/25) -- (6+4*6/25,5-4*8/25); 
        \draw[black,line width = 0.8pt, dotted]
            (8-3*6/25,6+3*8/25) -- (8+3*6/25,6-3*8/25);
        \draw[ultra thick, blue] 
            (0,0)  -- (1,0) -- (2,0) -- (2,1) -- (2,2) -- (3,2) -- (3,3) -- (4,3);
        \draw[ultra thick, blue]
            (4,3)  -- (5,3) -- (6,3) -- (6,4) -- (6,5);
        \draw[ultra thick, purple] 
            (6,5) -- (7,5) -- (7,6) -- (8,6);
        \draw[ultra thick, black] 
            (8,6) -- (9,6) -- (10,6) -- (11,6) -- (11,7) -- (12,7) -- (12,8) -- (13,8) -- (13,9) -- (13,10) -- (13,11) -- (14,11) -- (14,12) -- (15,12) -- (16,12);
        \fill[red]  
            (4,3) circle (2pt)
            (8,6) circle (2pt)
            (12,9) circle (2pt);
        \fill[teal]	
            (2,2) circle (4pt);
        \fill[teal]	
            (6,5) circle (4pt);
        \draw[shift = {(4.5,1.5)}, color = blue]        
            node[font=\Large] {$q_1$};
        \draw[shift = {(7.5,4.5)}, color = purple] 
            node[font=\Large] {$q_2$};
        \draw[shift = {(8.5,7.5)}, color = black] 
            node[font=\Large] {$r$};
        \draw[shift = {(5.5,5.5)}, color = teal] 
            node[font=\Large] {$L$};
    \end{tikzpicture}
    \caption{Canonical representation of a path $p$ in~$Y$.}
    \label{fig:Y_preimage}
\end{subfigure}
\quad \quad
%
%
\pgfmathsetmacro{\xShiftqOne}{0}
\pgfmathsetmacro{\yShiftqOne}{0}
\pgfmathsetmacro{\xShiftqTwo}{8}
\pgfmathsetmacro{\yShiftqTwo}{6}
\pgfmathsetmacro{\xShiftR}{-2}
\pgfmathsetmacro{\yShiftR}{-1}
\begin{subfigure}{.45\textwidth}
    \begin{tikzpicture}[scale = 0.45]
        \draw[gray!30] (0,0) grid (16,12);
        \draw[thick,red!50] (0,0) -- (16,12);
        \fill[opacity = 0.2, black]
            (\xShiftR+8,6+\yShiftR) -- (\xShiftR+8+6/25,6-8/25+\yShiftR) -- (\xShiftR+16+6/25, 12-8/25+\yShiftR) -- (\xShiftR+16,12+\yShiftR) -- cycle;
        \draw[black, line width = 1.5pt, opacity = 0.4]
            (\xShiftR+8+6/25,6-8/25+\yShiftR) -- (\xShiftR+16+6/25, 12-8/25+\yShiftR);
        \draw[black, line width = 1.5pt, opacity = 0.4]
            (\xShiftR+8,6+\yShiftR) -- (\xShiftR+16,12+\yShiftR);
        \draw[blue, line width = 1.5pt, dotted, opacity = 0.5]
            (0,1/2) -- (6,5);   
        \fill[opacity = 0.2, blue]
            (0,0) -- (0,1/2) -- (6,5) -- (6+6/25,5-8/25) -- cycle;
        \draw[blue, line width = 1.5pt, dotted, opacity = 0.5]
            (0,0) -- (6+ 6/25,5-8/25);
        \draw[purple, line width = 1.5pt, opacity = 0.5]
            (\xShiftqTwo+6,5+\yShiftqTwo) -- (\xShiftqTwo+8-2/3,6+\yShiftqTwo);    
        \fill[opacity = 0.2, purple]
            (\xShiftqTwo+6,5+\yShiftqTwo) -- (\xShiftqTwo+8-2/3,6+\yShiftqTwo) -- (\xShiftqTwo+8,6+\yShiftqTwo) -- (\xShiftqTwo+6+6/25,5-8/25+\yShiftqTwo) -- cycle;
        \draw[purple, line width = 1.5pt, dotted, opacity = 0.5]
            (\xShiftqTwo+6+ 6/25,5-8/25+\yShiftqTwo) -- (\xShiftqTwo+8,6+\yShiftqTwo);
        \draw[black,line width = 0.8pt, dotted]
            (6-2*6/25,5+2*8/25) -- (6+4*6/25,5-4*8/25); 
        \draw[black,line width = 0.8pt, dotted]
            (\xShiftqTwo+6-2*6/25,5+2*8/25+\yShiftqTwo) -- (\xShiftqTwo+6+4*6/25,5-4*8/25+\yShiftqTwo);
        \draw[ultra thick, blue] 
            (0,0)  -- (1,0) -- (2,0) -- (2,1) -- (2,2) -- (3,2) -- (3,3) -- (4,3);
        \draw[ultra thick, blue]
            (4,3)  -- (5,3) -- (6,3) -- (6,4) -- (6,5);
        \draw[ultra thick, purple] 
            (\xShiftqTwo+6,5+\yShiftqTwo) -- (\xShiftqTwo+7,5+\yShiftqTwo) -- (\xShiftqTwo+7,6+\yShiftqTwo) -- (\xShiftqTwo+8,6+\yShiftqTwo);
        \draw[ultra thick, black] 
            (\xShiftR+8,6+\yShiftR) -- (\xShiftR+9,6+\yShiftR) -- (\xShiftR+10,6+\yShiftR) -- (\xShiftR+11,6+\yShiftR) -- (\xShiftR+11,7+\yShiftR) -- (\xShiftR+12,7+\yShiftR) -- (\xShiftR+12,8+\yShiftR) -- (\xShiftR+13,8+\yShiftR) -- (\xShiftR+13,9+\yShiftR) -- (\xShiftR+13,10+\yShiftR) -- (\xShiftR+13,11+\yShiftR) -- (\xShiftR+14,11+\yShiftR) -- (\xShiftR+14,12+\yShiftR) -- (\xShiftR+15,12+\yShiftR) -- (\xShiftR+16,12+\yShiftR);
        \fill[red]  
            (4,3) circle (2pt)
            (8,6) circle (2pt)
            (12,9) circle (2pt);
        \fill[teal]	
            (2,2) circle (4pt);
        \fill[teal]	
            (6,5) circle (4pt);
        \fill[teal]	
            (\xShiftqTwo+6,5+\yShiftqTwo) circle (4pt);
        \draw[shift = {(4.5,1.5)}, color = blue] node[font=\Large] {$\mathbbm{q}_1$};
        \draw[shift = {(\xShiftqTwo+7.5,4.5+\yShiftqTwo)}, color = purple] node[font=\Large] {$\mathbbm{q}_2$};
        \draw[shift = {(\xShiftR+8.5,7.5+\yShiftR)}, color = black] node[font=\Large] {$\mathbbm{r}$};
        \draw[shift = {(5.5,5.5)}, color = teal] node[font=\Large] {$L$};
        \draw[shift = {(\xShiftqTwo+5.5,5.6+\yShiftqTwo)}, color = teal] node[font=\Large] {$L'$};
    \end{tikzpicture}
    \caption{Canonical representation of a path $\mathbbm{p}$ in~$\mathcal{Y}$.}
    \label{fig:Y_image}
\end{subfigure}
\caption{The maps $\phi^X$ and $\psi^{\mathcal{X}}$ act on the paths in diagrams (a) and (b), respectively, and their images are (b) and (a), respectively.  Similarly the maps $\phi^Y$ and $\psi^{\mathcal{Y}}$ act on paths in diagram~(c) and~(d), respectively, and their images are~(d) and~(c), respectively. 
The LPAs and HPBs of the paths determine open or closed forbidden regions (denoted using dotted or solid lines, respectively) within which no points of the path can lie. The labelling of the points in this figure is consistent (when applicable) with that used in the proofs of \cref{rem:codomain_phi_X,rem:codomain_phi_Y,rem:codomain_psi_X,rem:codomain_psi_Y}.}
\label{fig:XY_action}
\end{center}
\end{figure}

\subsubsection{The actions of \texorpdfstring{$\phi^X$ and $\phi^Y$}{phi\^X,phi\^Y}}

We now define the maps $\phi^X$ and~$\phi^Y$.  

\begin{definition}[Actions of $\phi^X$ and $\phi^Y$]
\phantomsection\label{def:phi}
Let $p \in W_k(g) \setminus S_k(g)$.
    \begin{description}

        \item[Case $p \in X$:]
            Write $p = q r_1 r_2$ according to \cref{def:domain_split}. Then $\phi^X: X \to \mathcal{X}$ is given by
            \begin{align}
                \phi^X(q r_1 r_2) = q r_2 r_1.
            \end{align}
    
        \item[Case $p \in Y$:]
            Write $p = q_1 q_2 r$ according to \cref{def:domain_split}. Then $\phi^Y: Y \to \mathcal{Y}$ is given by
            \begin{align}
                \phi^Y(q_1 q_2 r) = q_1 r q_2. 
            \end{align}
            \lrcornerqedequation
    \end{description}
\end{definition}

 \begin{prop} 
 \label{rem:codomain_phi_X}
 The map $\phi^X$ is well-defined.
 \end{prop}
 \begin{proof}
    Let $p\in X$. 
    We must check that 
    $\phi^X(p) = q r_2 r_1$ belongs to~$\mathcal{X}$.
    Let $H$ be the startpoint of~$r_2$.
    By \cref{def:domain_split}, $H$ is the last HPB of the path~$r_1 r_2$.
    Since $p$ is split at its last non-terminal boundary point into paths $q$ and $r_1 r_2$, we have:
    \begin{enumerate}
        \item 
        the path $r_1 r_2$ has no interior boundary points.
    \end{enumerate}
    Since $p \in X$, by \cref{def:domain_partition} we have:
    \begin{enumerate}
        \item[2.]
        the elevation of the LPAs of the path $q$ (if any) is greater than or equal to the magnitude of the elevation of~$H$ in the path~$r_1 r_2$.    
    \end{enumerate}

    It follows from statement 1 and \cref{obs:flaws} that:

    \begin{enumerate}
        \item[3.]
        the path $r_2 r_1$ has exactly one more flaw than does~$r_1 r_2$.
    \end{enumerate}

    It follows from statements 1 and 2 that:
    \begin{enumerate}
        \item[4.] 
        the subpath $r_2 r_1$ contains exactly one of the LPAs of~$\phi^X(p)$ (namely the endpoint of~$r_2$).
    \end{enumerate}

    It follows from statement 3 that $\phi^X(p)$ contains exactly one more flaw than~$p$. Furthermore, since the startpoint $H$ of $r_2$ is a boundary point of the path $\phi^X(p) = q r_2 r_1$, statement 4 implies that $\phi^X(p)$ cannot simultaneously satisfy both conditions \ref{def:codomain_partition:condition_1} and \ref{def:codomain_partition:condition_2} of \cref{def:codomain_partition}. Therefore $\phi^X(p) \in \mathcal{X}$, as required.
\end{proof}

\begin{remark}
\label{rem:remark_after_phi_X_well-defined}
    Continue with the notation from the proof of \cref{rem:codomain_phi_X}.
    We note for use in \cref{subsec:proof} that, since the startpoint $H$ of $r_2$ is the last HPB of the path $r_1r_2$ and is a boundary point of the path~$\phi^X(p)$, we have that $H$ is the only boundary point of $\phi^X(p) = qr_2r_1$ contained in the subpath~$r_2$.\lrcornerqed
\end{remark}

\begin{prop}
The map $\phi^Y$ is well-defined.
\label{rem:codomain_phi_Y}
\end{prop}
\begin{proof}
    Let $p\in Y$. 
    We must check that 
    $\phi^Y(p) = q_1 r q_2$ belongs to~$\mathcal{Y}$.
    Let $L$ be the endpoint of~$q_1$.
    By \cref{def:domain_split}, $L$ is the last LPA of the path~$q_1 q_2$.
    Since $p \in Y$, by \cref{def:domain_partition} the elevation of the LPAs of the path $q_1 q_2$ (including~$L$) is smaller than the magnitude of the elevation of the HPBs of the path~$r$ (if any). Therefore:
    \begin{enumerate}
        \item the path $\phi^Y(p)$ contains exactly one more flaw than $p$, namely the startpoint $L'$ of~$q_2$.
        \item the points $L$ and $L'$ are the (distinct) last two LPAs of $\phi^Y(p)$ (since the path $r$ has no interior boundary points by \cref{def:domain_split}). 
        \item the subpath $r$ contains no boundary points of~$\phi^Y(p)$.
    \end{enumerate}
    This shows by \cref{def:codomain_partition} that $q_1 r q_2 \in \mathcal{Y}$, as required.
\end{proof}

\subsubsection{The actions of \texorpdfstring{$\psi^\mathcal X$ and $\psi^\mathcal Y$}{psi\^X and psi\^Y}}
\phantomsection
\label{subsec:psi_action}

We now define the maps $\psi^\mathcal X$ and $\psi^\mathcal Y$.  

\begin{definition}[Actions of $\psi^\mathcal X$ and $\psi^\mathcal Y$]
\phantomsection
\label{def:psi}
Let $\mathbbm{p} \in W_{k+1}(g)$.
    \begin{description}
         \item[Case $\mathbbm{p} \in \mathcal{X}$:]
            Write $\mathbbm{p} = \mathbbm{q} \mathbbm{r}_2 \mathbbm{r}_1$ according to \cref{def:codomain_split}. Then $\psi^\mathcal X: \mathcal{X} \to X$ is given by 
            \begin{align}
            \psi^\mathcal X(\mathbbm{q} \mathbbm{r}_2 \mathbbm{r}_1) = \mathbbm{q} \mathbbm{r}_1 \mathbbm{r}_2.
        \end{align}

        \item[Case $\mathbbm{p} \in \mathcal{Y}$:]
            Write $\mathbbm{p} = \mathbbm{q}_1 \mathbbm{r} \mathbbm{q}_2$ according to \cref{def:codomain_split}. Then $\psi^\mathcal Y: \mathcal{Y} \to Y$ is given by 
            \begin{align}
            \psi^\mathcal Y(\mathbbm{q}_1 \mathbbm{r} \mathbbm{q}_2) = \mathbbm{q}_1 \mathbbm{q}_2 \mathbbm{r}.
             \end{align}
        \lrcornerqedequation
    \end{description}
\end{definition}

\begin{prop}
\label{rem:codomain_psi_X}
The map $\psi^\mathcal{X}$ is well-defined.
\end{prop}
\begin{proof}
    Let $\mathbbm{p} \in \mathcal{X}$.
    We must check that 
    $\psi^\mathcal X(\mathbbm{p}) = \mathbbm{q} \mathbbm{r}_1 \mathbbm{r}_2$ belongs to~$X$.

    Let $L$ be the endpoint of the path~$\mathbbm{r}_2$.
    By \cref{def:codomain_split}, we have:

    \begin{enumerate}
    \item[1.]
    $L$ is the last LPA of $\mathbbm{p}= \mathbbm{q} \mathbbm{r}_2 \mathbbm{r}_1$, and the startpoint of $\mathbbm{r}_2$ is the boundary point of $\mathbbm{p}$ which immediately precedes~$L$.
    \end{enumerate}
    
    By \cref{def:codomain_partition},  we have:

    \begin{enumerate}
    \item[2.]
    either $\mathbbm{p}$ has exactly one LPA,
    or the subpath of $\mathbbm{p}$ lying between the last two LPAs of $\mathbbm{p}$ contains a boundary point of~$\mathbbm{p}$.
    \end{enumerate}
    
    It follows from statements 1 and 2 that:

    \begin{enumerate}
    \item[3.]
    the subpath $\mathbbm{r}_2 \mathbbm{r}_1$ of $\mathbbm{p}$ contains exactly one LPA of $\mathbbm{p}$, namely the point~$L$.
    \end{enumerate}
    
    Statement 3 and \cref{obs:flaws} imply that:

    \begin{enumerate}
    \item[4.]
    the path $\psi^{\mathcal X}(\mathbbm{p}) = \mathbbm{q} \mathbbm{r}_1 \mathbbm{r}_2$ splits at its last non-terminal boundary point into the paths $\mathbbm{q}$ and~$\mathbbm{r}_1 \mathbbm{r}_2$.
    \item[5.]
    the path $\psi^\mathcal X(\mathbbm{p})= \mathbbm{q} \mathbbm{r}_1 \mathbbm{r}_2$ has exactly one fewer flaw than $\mathbbm{p}$.
    \end{enumerate}
    The elevation of $L$ in the path~$\mathbbm{r}_2 \mathbbm{r}_1$ equals the 
    magnitude of the elevation of the HPBs of the path~$\mathbbm{r}_1 \mathbbm{r}_2$. Since $L$ is an LPA of $\mathbbm{p}$ by statement 1, this gives:

    \begin{enumerate}
    \item[6.]
    the elevation of the LPAs of the path~$\mathbbm q$ (if any) is greater than or equal to the magnitude of the elevation of the HPBs of the path~$\mathbbm{r}_1 \mathbbm{r}_2$.
    \end{enumerate}
    
    Statements 4, 5 and 6 show by \cref{def:domain_partition} that $\psi^\mathcal X(\mathbbm{p}) \in X$.
\end{proof}

\begin{remark}
\label{rem:remark_after_psi_X_well-defined}
    Continue with the notation from the proof of \cref{rem:codomain_psi_X}.
    We note for use in \cref{subsec:proof} that statement~1 implies        
        the endpoint of $\mathbbm r_1$ is the last HPB of the path $\mathbbm r_1 \mathbbm r_2$. \lrcornerqed
\end{remark}

\begin{prop}
The map $\psi^\mathcal Y$ is well-defined.
\label{rem:codomain_psi_Y}
\end{prop}
\begin{proof}
    Let $\mathbbm{p} \in \mathcal{Y}$.
    We must check  
    $\psi^\mathcal Y(\mathbbm{p}) = \mathbbm{q}_1 \mathbbm{q}_2 \mathbbm{r}$ belongs to~$Y$.
    Let $L$ be the endpoint of the path~$\mathbbm{q}_1$, and let $L'$ be the startpoint of the path~$\mathbbm{q}_2$.
    By \cref{def:codomain_split}, we have:

    \begin{enumerate}
    \item[1.]
    \label{rem:codomain_psi_Y:L_and_L'_are_LPAs}
    $L$ and $L'$ are the last two LPAs of~$\mathbbm p = \mathbbm{q}_1 \mathbbm{r} \mathbbm{q}_2$.
    \end{enumerate}

    It follows from statement 1 that
    \begin{enumerate}
    \item[2.]
    the path $\psi^{\mathcal Y}(\mathbbm{p}) = \mathbbm{q}_1 \mathbbm{q}_2 \mathbbm{r}$ splits at its last non-terminal boundary point into the paths $\mathbbm{q}_1 \mathbbm{q}_2$ and~$\mathbbm{r}$.
    \end{enumerate}
    
    The LPAs $L$ and $L'$ of $\mathbbm{p}$ combine to form a single point in $\psi^\mathcal Y(\mathbbm{p}) = \mathbbm{q}_1 \mathbbm{q}_2 \mathbbm{r}$, and so:
    
    \begin{enumerate}
    \item[3.]
    the subpaths $\mathbbm{q}_1$ and $\mathbbm{q}_2$ of $\psi^\mathcal Y(\mathbbm{p})$ collectively contain exactly one fewer flaw than the subpaths $\mathbbm{q}_1$ and $\mathbbm{q}_2$ of~$\mathbbm{p}$.
    
    \item[4.]
    the path $\mathbbm{q}_1 \mathbbm{q}_2$ has at least one LPA, namely the point~$L=L'$.
    \end{enumerate}

    By statement 1 and \cref{def:codomain_partition}\ref{def:codomain_partition:condition_2}, the subpath $\mathbbm r$ contains no boundary points of~$\mathbbm p$. This, together with statement 4, implies:

    \begin{enumerate}
    \item[5.]
    the elevation of the LPAs of the path $\mathbbm{q}_1 \mathbbm{q}_2$ is smaller than the magnitude of the elevation of the HPBs of the path~$\mathbbm{r}$ (if any).
    \end{enumerate}
    It follows from statement 5 that:
    \begin{enumerate}
    \item[6.]
    disregarding its startpoint and endpoint, the subpath $\mathbbm{r}$ of $\psi^\mathcal Y(\mathbbm{p}) = \mathbbm{q}_1 \mathbbm{q}_2 \mathbbm{r}$  contains the same number of flaws as the subpath $\mathbbm{r}$ of~$\mathbbm{p}$.
    \end{enumerate}
    
    By statements 3 and 6, the path $\psi^{\mathcal Y}(\mathbbm{p})$ contains exactly one fewer flaw than~$\mathbbm{p}$. By statement 4, the path $\mathbbm q_1 \mathbbm q_2$ has at least one flaw. Together with statements 2 and 5, this shows by \cref{def:domain_partition} that $\psi^\mathcal Y(\mathbbm{p}) \in Y$.
\end{proof}

\begin{remark}
\label{rem:remark_after_psi_Y_well-defined}
    Continue with the notation from the proof of \cref{rem:codomain_psi_Y}.
    We note for use in \cref{subsec:proof} that statement~1 implies the endpoint $L$ of $\mathbbm q_1$ is the last LPA of the path $\mathbbm{q}_1 \mathbbm{q}_2$. \lrcornerqed
\end{remark}

\subsection{The maps \texorpdfstring{$\phi$}{phi} and \texorpdfstring{$\psi$}{psi} are injective}
\phantomsection
\label{subsec:proof}

We complete the proof of \cref{thm:set_sizes} by showing in turn that each of the maps $\phi^X$, $\phi^Y$, $\psi^{\mathcal{X}}$, $\psi^{\mathcal{Y}}$ is injective.
We give the proof for $\phi^X$ and $\phi^Y$ in detail, and for $\psi^{\mathcal{X}}$ and $\psi^{\mathcal{Y}}$ in abbreviated form.

\begin{description}
    \item[The map $\phi^X$ is injective:]\hfill \\[6pt]
        Let $p,p' \in X$, and write $p = q r_1 r_2$ and $p' = q' r_1' r_2'$ according to \cref{def:domain_split}. We suppose that $\phi^X(p) = \phi^X(p')$, and wish to show that $p=p'$.

        By statement~4 in the proof of \cref{rem:codomain_phi_X},
        the endpoint $L$ of $r_2$ is the last LPA of~$\phi^X(p) = q r_2 r_1$.
        By \cref{rem:remark_after_phi_X_well-defined},
        the startpoint $H$ of $r_2$ is the boundary point of $\phi^X(p) = qr_2r_1$ immediately preceding~$L$.

        Therefore $\phi^X(p) = qr_2r_1$ splits into $qr_2$ and $r_1$ at the last LPA $L$ of~$\phi^X(p)$, and the subpath $qr_2$ splits into $q$ and~$r_2$ at the boundary point of $\phi^X(p)$ immediately preceding~$L$. The corresponding statement holds for $\phi^X(p')$. Since $\phi^X(p)$ and $\phi^X(p')$ are equal by assumption, their LPAs and boundary points are identical. 
        Therefore $q = q'$ and $r_2=r_2'$ and $r_1 = r_1'$ and so 
        $p = q r_1 r_2 = q' r_1' r_2' = p'$, as required.

    \item[The map $\phi^Y$ is injective:]\hfill \\[6pt]
        Let $p, p' \in Y$, and write $p = q_1 q_2 r$ and $p' = q_1' q_2' r'$ according to \cref{def:domain_split}. 
        We suppose that $\phi^Y(p) = \phi^Y(p')$,
        and wish to show that $p = p'$. 

        By statement~2 in the proof of \cref{rem:codomain_phi_Y}, the endpoint $L$ of $q_1$ and the startpoint $L'$ of $q_2$
        are the last two LPAs of~$\phi^Y(p)=q_1 r q_2$, and so  $\phi^Y(p)$ splits at its last two LPAs into $q_1$ and $r$ and $q_2$.
        Likewise, $\phi^Y(p')$ splits at its last two LPAs into $q_1'$ and $r'$ and $q_2'$. But $\phi^Y(p)$ and $\phi^Y(p')$ are equal by assumption, so their last two LPAs are identical. Therefore $q_1 = q_1'$ and $r=r'$ and $q_2 = q_2'$ and so 
        $p = q_1 q_2 r = q_1' q_2' r' = p'$, as required.
        
    \item[The map $\psi^\mathcal X$ is injective:]\hfill \\[6pt]
        Let $\mathbbm{p}, \mathbbm {p'} \in \mathcal X$, and write $\mathbbm p = \mathbbm q \mathbbm r_2 \mathbbm r_1$ and $\mathbbm p' = \mathbbm q' \mathbbm r_2' \mathbbm r_1'$ according to \cref{def:codomain_split}. 
        We suppose that $\psi^\mathcal X(\mathbbm p) = \psi^\mathcal X(\mathbbm p')$, and wish to show that $\mathbbm p=\mathbbm p'$.

        By statement~4 in the proof of \cref{rem:codomain_psi_X},
        the path $\psi^{\mathcal X}(\mathbbm{p}) = \mathbbm{q} \mathbbm{r}_1 \mathbbm{r}_2$ splits at its last non-terminal boundary point into $\mathbbm{q}$ and~$\mathbbm{r}_1 \mathbbm{r}_2$.
        By \cref{rem:remark_after_psi_X_well-defined}
        the path $\mathbbm{r}_1 \mathbbm{r}_2$ splits at its last HPB into $\mathbbm{r}_1$ and~$\mathbbm{r}_2$.
        
        It follows that 
        $\mathbbm{p} = \mathbbm{q} \mathbbm{r}_2 \mathbbm{r}_1 = \mathbbm{q}' \mathbbm{r}_2' \mathbbm{r}_1' = \mathbbm{p}'$, as required.

    \item[The map $\psi^\mathcal Y$ is injective:]\hfill \\[6pt]
        Let $\mathbbm{p}, \mathbbm {p'} \in \mathcal Y$, and write $\mathbbm p = \mathbbm q_1 \mathbbm r \mathbbm q_2$ and $\mathbbm p' = \mathbbm q_1' \mathbbm r' \mathbbm q_2'$ according to \cref{def:codomain_split}. 
        We suppose that $\psi^\mathcal Y(\mathbbm p) = \psi^\mathcal Y(\mathbbm p')$, and wish to show that $\mathbbm p=\mathbbm p'$.

        By statement~2 in the proof of \cref{rem:codomain_psi_Y}, the path $\psi^{\mathcal Y}(\mathbbm{p}) = \mathbbm{q}_1 \mathbbm{q}_2 \mathbbm{r}$ splits at its last non-terminal boundary point into $\mathbbm{q}_1 \mathbbm{q}_2$ and~$\mathbbm{r}$.
        By \cref{rem:remark_after_psi_Y_well-defined},
        the path $\mathbbm{q}_1 \mathbbm{q}_2$ splits at its last LPA into $\mathbbm{q}_1$ and~$\mathbbm{q}_2$.
        
        It follows that 
        $\mathbbm{p} = \mathbbm{q}_1 \mathbbm{r} \mathbbm{q}_2 = \mathbbm{q}_1' \mathbbm{r}' \mathbbm{q}_2' = \mathbbm{p}'$, as required.

\end{description}

\subsection{The special case \texorpdfstring{$g=1$}{geq1} \texorpdfstring{(\cref{thm:fir_mar_rat})}{}}
\phantomsection
\label{sec:geqaul1}

We finally re-examine the special case $g=1$ (\cref{thm:fir_mar_rat}), involving paths from $(0,0)$ to $(a,b)$, to show how the proof of \cref{thm:set_sizes} described in \cref{sec:bijection} simplifies significantly.   In doing so, we shall obtain a simple self-contained proof of \cref{thm:fir_mar_rat}.

Let $k$ satisfy $0\le k < a+b-1$ and let $p$ be a path in~$W_k(1)$. Since
\begin{enumerate}
    \item 
    the path $p$ cannot have any interior boundary points,
\end{enumerate}
it follows that:
\begin{enumerate}
    \item[2.]
    the subset $S_k(1)$ of $W_k(1)$ is empty by \cref{def:S}, 

    \item[3.]
    the subset $Y$ of  $ W_k(1)$ is empty by \cref{def:domain_partition}.
\end{enumerate}

Since $g=1$, the path $p$ has at most one LPA and so
\begin{enumerate}
    \item[4.]
    the subset $\mathcal Y$ of $ W_{k+1}(1)$ is empty by \cref{def:codomain_partition}.
\end{enumerate}
By reference to \cref{fig:phi_construction_new}, statements 2, 3 and 4 show that $W_k(1) = X$ and $W_{k+1}(1) = \mathcal{X}$.
We shall show that $\phi^X$ and $\psi^X$ are inverse maps, so that $|W_k(1)| = |W_{k+1}(1)|$. We may then conclude that $|W_0(1)| = |W_1(1)| = \cdots = |W_{a+b-1}(1)|$, which gives \cref{thm:fir_mar_rat} because the total number of paths from $(0,0)$ to $(a,b)$ is~$\binom{a+b}{a}$.

It remains to show that $\phi^X$ and $\psi^X$ are inverse maps. Write $p \in X$ as its canonical representation $p=qr_1r_2$ according to \cref{def:domain_split}, where $p$ has at least one HPB. 
By statement 1, we have that $q$ is empty and so $p=r_1r_2$. Since~$g=1$, the path $p$ has at most one HPB.
Therefore by \cref{def:phi} the map $\phi^X$ splits $p$ at its unique HPB into $r_1r_2$ and replaces it by~$r_2r_1$. That is, $\phi^X$  cyclically permutes the steps of $p$ by bringing the unique HPB to the origin.

Similarly, write $\mathbbm p \in \mathcal X$ as its canonical representation $\mathbbm p = \mathbbm q \mathbbm r_2 \mathbbm r_1$ according to \cref{def:codomain_split}, where $\mathbbm q$ is empty by statement~$1$. 
By \cref{def:psi}, the map $\psi^X$ splits $\mathbbm p$ at its unique LPA into $\mathbbm r_2 \mathbbm r_1$ and replaces it by~$\mathbbm r_1 \mathbbm r_2$. That is, $\psi^X$  cyclically permutes the steps of $p$ by bringing the unique LPA to the origin. 

Comparison of the descriptions of $\phi^X$ and $\psi^X$ shows that they are inverse maps, as required.

\section{Conclusion}
Our central objective was to find an explicit formula for~$|W_k(g)|$, the 
number of simple lattice paths from $(0,0)$ to~$(ga,gb)$ having exactly $k$ lattice points lying strictly above the linear boundary joining the startpoint to the endpoint. This is given by the closed form expression in \cref{thm:path_counting_formula}, using the definition \eqref{def:mu_j} of~$\mu_j(g)$.

We conclude with two open problems for future study.

\begin{enumerate}
    \item 
    Evaluating $|W_k(g)|$ via the path enumeration formula \cref{thm:path_counting_formula} involves a sum over integer partitions of~$g$, and is therefore computationally intensive.  In the special case $a=1$, \cref{thm:path_counting_aisone_case} provides an alternative expression to \cref{thm:path_counting_formula} that is computationally simpler. Is there a closed form expression for~$|W_k(g)|$ that is computationally simpler than \cref{thm:path_counting_formula} for other special cases of~$a,b$ (or in general)?
    
    \item We established the path enumeration formula by solving the recurrence relation given in \cref{cor:rec} and making use of the known values stated in \cref{thm:bizley}. These known values are in turn predicated on \cref{thm:actual_bizley}, which was proved by Bizley using generating functions~\cite{biz}. 
    Is there a direct combinatorial proof of \cref{thm:actual_bizley}?
\end{enumerate}

\section{Acknowledgments}
The first and third authors collaborated with Takudzwa Marwendo in 2019 on a preliminary investigation of the material of this paper, resulting in \cref{thm:fir_mar_rat,cor:path_counting_aisbisone_case}, the results of various numerical experiments, and the conjecture that 
property \hyperref[property:const_blocks]{P1} (constant on blocks) 
holds. This investigation was a major inspiration for this paper. We are grateful to Taku for his many contributions and helpful conversations. We thank Ira Gessel for helpful conversations, Fran\c{c}ois Bergeron for pointing out the reference~\cite{triangular-berg}, and Alvaro Gutierrez Caceres for some stimulating discussions.
We thank the referee for their very careful reading of the paper and for providing many helpful suggestions for improvement.


\end{document}